\documentclass[11pt,a4paper]{article}
\newcommand{\ratio}{\nu, L}
\def\mean#1{\mathchoice%
          {\mathop{\kern 0.2em\vrule width 0.6em height 0.69678ex depth -0.58065ex
                  \kern -0.8em \intop}\nolimits_{\kern -0.4em#1}}%
          {\mathop{\kern 0.1em\vrule width 0.5em height 0.69678ex depth -0.60387ex
                  \kern -0.6em \intop}\nolimits_{#1}}%
          {\mathop{\kern 0.1em\vrule width 0.5em height 0.69678ex
              depth -0.60387ex
                  \kern -0.6em \intop}\nolimits_{#1}}%
          {\mathop{\kern 0.1em\vrule width 0.5em height 0.69678ex depth -0.60387ex
                  \kern -0.6em \intop}\nolimits_{#1}}}

\newcommand{\tw}{\tilde{w}}

\def\en{\mathbb N}
\def\er{\mathbb R}
\newcommand{\ern}{\mathbb{R}^n}

\def\loc{\operatorname{loc}}

\usepackage[nottoc]{tocbibind}
\usepackage{graphicx}
\usepackage{enumerate}
\usepackage{amsmath}
\usepackage{amsfonts}
\usepackage{psfrag}
\usepackage{color}
\usepackage{pgf,tikz}
\usepackage{mathrsfs}
\usetikzlibrary{arrows}
\usepackage{fancyhdr}
\usepackage{txfonts} 
\newcommand{\data}{\texttt{data}}

\usepackage[colorlinks,pdfpagelabels,pdfstartview = FitH,bookmarksopen = true,bookmarksnumbered = true,linkcolor = blue,plainpages = false,hypertexnames = false,citecolor = red] {hyperref}

\newtheorem{theorem}{Theorem}
\def\loc{\operatorname{loc}}

\newtheorem{lemma}{Lemma}

\newcommand{\divo}{\textnormal{div}}

\newcommand{\mint}{\mathop{\int\hskip -1,05em -\, \!\!\!}\nolimits}

\numberwithin{equation}{section}
\allowdisplaybreaks[4]

\newenvironment{proof}{\smallskip\noindent\emph{Proof.}\hspace{1pt}}%
{\hspace{-5pt}{\nobreak\quad\nobreak\hfill\nobreak$\square$\vspace{8pt}%
\par}\smallskip\goodbreak}

\def\en{\mathbb N}
\def\er{\mathbb R}

\newcommand\eps\varepsilon
\def\eqn#1$$#2$${\begin{equation}\label#1#2\end{equation}}

\newcommand{\be}{\begin{equation}}
\newcommand{\ee}{\end{equation}}
\newcommand{\osc}{\operatorname{osc}}

\newcommand{\supp}{\operatorname{supp}}

\newcommand{\dist}{\operatorname{dist}}

\newcommand{\snr}[1]{\lvert #1\rvert}

\newcommand{\nr}[1]{\lVert #1 \rVert}

\newcommand{\diver}{\operatorname{div}}

\newcommand{\N}{\mathbb{N}}

\def\name[#1, #2]{#1 #2}

\newcommand{\rif}[1]{(\ref{#1})}
\newcommand{\trif}[1] {\textnormal{\rif{#1}}}

\newcommand{\stackleq}[1]{\stackrel{\rif{#1}}{ \leq}}

\definecolor{ffqqqq}{rgb}{1.,0.,0.}
\definecolor{uuuuuu}{rgb}{0.26666666666666666,0.26666666666666666,0.26666666666666666}

\begin{document}

\title{A borderline case of Calder\'on-Zygmund estimates for non-uniformly elliptic problems}

\author{
\textsc{Cristiana De Filippis\thanks{Mathematical Institute, University of Oxford, Andrew Wiles Building, Radcliffe Observatory Quarter, Woodstock Road, Oxford, OX26GG, Oxford, United Kingdom. E-mail:
      \texttt{Cristiana.DeFilippis@maths.ox.ac.uk}} \ and   Giuseppe Mingione \thanks{Dipartimento SMFI, Universit\`a di Parma, Viale delle Scienze 53/a, Campus, 43124 Parma, Italy
E-mail: \texttt{giuseppe.mingione@unipr.it}}}}

\date{\today
}

\maketitle
\thispagestyle{plain}

\vspace{-0.6cm}
\begin{abstract}
We show, in a borderline case which was not covered before, the validity of nonlinear Calder\'on-Zygmund estimates for a class of non-uniformly elliptic problems driven by double phase energies. 
\end{abstract}

       \vspace{5pt}                           %
\vspace{0.2cm}

\setlength{\voffset}{-0in} \setlength{\textheight}{0.9\textheight}

\setcounter{page}{1} \setcounter{equation}{0}
\section{Introduction}

The aim of this paper is to complete the Calder\'on-Zygmund type theory obtained in \cite{CM3}, achieving a delicate borderline case that has been left open there. Let us briefly summarize the situation. In \cite{BCM1, BCM3, CM1, CM2} the authors have provided a basic regularity theory for minimizers of double phase functionals of the type  
\eqn{split}
$$
\mathcal P(v, \Omega) :=\int_{\Omega} (|Dv|^p+a(x)|Dv|^q) \, dx\;, $$
where
\eqn{ass1}
$$
1<  p <  q\,, \qquad \qquad 0\leq a(\cdot) \in C^{0,\alpha}(\Omega)\,, \qquad \qquad \alpha \in (0,1]\,.
$$
Here, as in the rest of the paper, $\Omega\subset \er^n$ is a bounded open subset of $\er^n$ and $n\geq 2$; we refer to Section \ref{sezionedue} below for more notation. The functional $\mathcal P$ is characterized by the fact that it changes the rate of ellipticity according to the positivity of the coefficient $a(x)$: when $a(x)>0$ the integrand has $q$-polynomial behaviour with respect to the gradient, otherwise it shows $p$-polynomial one.  Needless to say, when $a(x)\equiv 0$ the functional $\mathcal P$ reduces to the standard $p$-Laplacean functional (we refer to the papers \cite{KuMi1, KuMi2} for a recent update of regularity theory in the standard $p$-case and to \cite{Uh, Ur} for the beginnings). This kind of functional has been introduced by Zhikov in a series of remarkable papers \cite{Z1, Z2, Z3}. He was motivated by speculations on theoretical aspects of the Calculus of Variations such as the Lavrentiev phenomenon, and by some problems arising in the Homogenization of composite and strongly anisotropic materials. We refer the reader to the introductory sections of \cite{CM1, CM2}, and especially, of \cite{BCM2}, for a comprehensive discussion on the subject and its role in the setting of non-uniformly elliptic problems and modern regularity theory. 

Recently, Colombo and the second-named author of the present paper have investigated the validity of Calder\'on-Zygmund estimates for solutions to non-homogenous equations that are naturally connected to the Euler-Lagrange equations of the functional in \rif{split}. The model equations is in this case given by 
\eqn{basic-E}
$$
-\divo\, (p|Du|^{p-2}Du + a(x)q|Du|^{q-2}Du)=-\divo\, (|F|^{p-2}F +  a(x)|F|^{q-2}F)\;.
$$
which is in fact the Euler-Lagrange equation of the functional 
$$
v \to \mathcal P(v, \Omega) -  \int_{\Omega} \langle |F|^{p-2}F+a(x)|F|^{q-2}F, Dv\rangle\, dx
$$
where $F \colon \Omega \to \er^n$ is a given vector field. For this situation one of the main results of \cite{CM3} claims the validity of the sharp implication 
\eqn{implica}
$$
|F|^{p} + a(x)|F|^{q} \in L^\gamma_{\loc}(\Omega) \Longrightarrow |Du|^{p} + a(x)|Du|^{q}\in L^\gamma_{\loc}(\Omega) \qquad \forall \ \gamma >1\;,
$$
under the main assumption 
\eqn{ass3}
$$
\frac{q}{p}< 1+\frac{\alpha}{n}\;.
$$
As in fact shown in \cite{ELM}, \rif{implica} fails to hold in the case $
q/p> 1+\alpha/n.
$ In this paper we show that \rif{implica} still holds in the delicate limiting case $q/p = 1+\alpha/n$, thereby replacing \rif{ass3} by 
\eqn{ass4}
$$
\frac{q}{p}\leq 1+\frac{\alpha}{n}\;.
$$
Borderline cases are always delicate, and, indeed, in this paper we shall exploit a few subtle facts that have been overlooked in \cite{CM3}. We again refer to the Introduction of \cite{CM3} for a description of the problems concerning the type of results in \rif{implica} and their place in what is nowadays called Nonlinear Calder\'on-Zygmund theory. We just confine ourselves to remark that, in view of the example in \cite{ELM}, the condition in \rif{ass4} is necessary to obtain the main result of this paper, that is Theorem \ref{cz} below. We also remark that, in the standard case $a(\cdot)\equiv 0$ or $p=q$, the results in \rif{implica} gives back the classical Calder\'on-Zygmund estimate for $p$-Laplacean type operators. 

Our result here actually holds for a class of equations that are more general of the one in \rif{basic-E} -- and in fact this has already been dealt with in \cite{CM3}. We shall consider equations of the type
\begin{flalign}\label{eq}
-\diver A(x,Du)=-\diver G(x,F) \qquad \mathrm{in} \ \Omega \subset \er^n\;.
\end{flalign}
The assumptions on the continuous vector field $A \colon \Omega \times \er^n\to \er^n$ are now as follows:
\begin{flalign}\label{assa}
\begin{cases}
\ A\in C(\Omega\times \mathbb{R}^{n},\mathbb{R}^{n}) \ \ \mathrm{and} \ \ z\mapsto A(\cdot,z)\in C^{1}(\mathbb{R}^{n}\setminus\{0\},\mathbb{R}^{n})\\
\ \snr{A(x,z)}+\snr{\partial A(x,z)}\snr{z}\le L\left(\snr{z}^{p-1}+a(x)\snr{z}^{q-1}\right)\\
\ \nu\left(\snr{z}^{p-2}+a(x)\snr{z}^{q-2}\right)\snr{\xi}^{2}\le \langle \partial A(x,z)\xi,\xi\rangle\\
\ \snr{A(x_{1},z)-A(x_{2},z)}\le L\snr{a(x_{1})-a(x_{2})}\snr{z}^{q-1}
\end{cases}\;,
\end{flalign}
for all $z\in \mathbb{R}^{n}\setminus \{0\}$, $\xi \in \mathbb{R}^{n}$, $x,x_{1},x_{2}\in \Omega$, where $0 < \nu\leq L < \infty$ are fixed ellipticity constants. The coefficient $a(\cdot)$ and the numbers $p,q, \alpha$ have already been specified in \rif{ass1}. As for the right-hand side, the Carath\'eodory regular vector field $G \colon \Omega \times \er^n\to \er^n$ is instead assumed to verity the following natural growth conditions:
\begin{flalign}\label{G}
\snr{G(x,z)}\le L\left(\snr{z}^{p-1}+a(x)\snr{z}^{q-1}\right)\;.
\end{flalign}
Assumptions \rif{assa}-\rif{G} perfectly cover the case displayed in \rif{basic-E} and are in fact modelled on them. Finally, before going on, let us mention that in the rest of the paper we shall use the compact notation
\eqn{dataref} 
$$\data \equiv 	\data \left(n,p,q,\alpha,\ratio,\|a\|_{L^\infty}, [a]_{\alpha}, \|H(\cdot,Du)\|_{L^1(\Omega)}\right)\;,$$
including several data of the problem considered. In the rest of the paper we shall denote 
\eqn{defiH}
$$
H(x, z):= |z|^p + a(x)|z|^q
$$
whenever $x\in \Omega$ and $z\in \er^n$; we shall, with some abuse of notation, keep on denoting as in \rif{defiH} also in the case $z\in \er$.  
We then have the following main result of the paper:
\begin{theorem}\label{cz}
Let $u\in W^{1,1}(\Omega)$ be a distributional solution to \trif{eq}, such that $H(x,Du), H(x,F) \in L^1(\Omega)$ and under the assumptions \trif{ass1}, \trif{ass4}, \trif{assa}, \trif{G}. Then \trif{implica} holds. 
Moreover, fix open subsets $\Omega_0\Subset \tilde \Omega_0\Subset \Omega$ with $\dist (\Omega_0 , \partial \tilde \Omega_0)\approx \dist (\tilde \Omega_0 , \partial \Omega)\approx \dist (\Omega_0 , \partial  \Omega)$; for every $\gamma > 1$, there exist a radius 
$r>0$ and a constant $c\geq 1$, both depending on $\textnormal{\data}, \dist ( \dist (\Omega_0, \partial \Omega), \gamma$ and $ \|F\|_{L^\gamma(\tilde \Omega_0)}$, such that the inequality
\eqn{final-estimate}
$$
\left(\mean{B_{\varrho/2}} [H(x,Du)]^\gamma \, dx\right)^{1/\gamma}  \leq  c\mean{B_{\varrho}} H(x,Du) \, dx  +
c  \left(\mean{B_{\varrho}} [H(x,F)]^\gamma \, dx\right)^{1/\gamma}
$$
holds for every ball $B_\varrho \subset \Omega_0$ such that $\varrho \leq r$.\end{theorem}
We remark that the previous theorem has been obtained in \cite{CM3} assuming \rif{ass3}. In this paper, starting from the approach considered in \cite{CM3}, we are able to achieve the limiting case in \rif{ass4} by using an improved approach of the fractional estimates originally developed in \cite{CM1}. These are estimates aimed at proving that the gradient of solutions to homogeneous equations as
\begin{flalign}\label{eq-hom}
-\diver A(x,Dw)=0 \;,
\end{flalign} 
under assumptions as \rif{ass4} and \rif{assa} belong to suitable fractional Sobolev spaces. See Theorem \ref{T1} below. Let us remark that fractional differentiability properties of solutions to various types of potentially degenerate elliptic equations are a powerful tool in regularity theory (see for instance \cite{Min2} where these are employed to estimate singular sets of solutions), and have been recently the object of investigation in different settings for both local and nonlocal operators including rough data too \cite{AKM, KuMi1, KuMi3}. The improvement in the arguments of \cite{CM1} then comes from a further application of a preliminary higher integrability result for solutions to \rif{eq-hom}, that uses certain classical self-improving properties of reverse H\"older inequalities. Once this improvement is reached we can revisit the proof given in \cite{CM3} to get the statement of Theorem \ref{cz}. Let us finally observe that the type of problems considered in this paper are related to so called functionals with $(p,q)$-growth conditions. These have been extensively treated in the literature over the last years starting by the papers of Marcellini \cite{M1, M2}. Eventually, several contributions have been given in this direction, see for instance \cite{DF1, DF2, DO, H1, H2, H3, lieb1, lieb2, OK1, OK2, RT2, TU} amongst the most closely related to the setting we are considering, that is the one of non-autonomous functionals with non-standard growth conditions.

\section{Notation and preliminaries}\label{sezionedue}
In the rest of the paper we shall denote by $c$ a general positive constant, possibly varying from line to line; special occurrences will be denoted by $c_1,  c_*, \bar c$ or the like. All such constants will always be {\em larger or equal than one}; relevant
dependencies on parameters will be emphasised using parentheses, i.e., ~$c_{1}\equiv c_1(n,p,q)$ means that $c_1$ 
depends on $n,p,q$. We denote by $ B_r(x_0):=\{x \in \er^n \, : \,  |x-x_0|< r\}$ the open ball with center $x_0$ and radius 
$r>0$; when no confusion will arise, we shall omit denoting the centre as follows: $B_r \equiv B_r(x_0)$. Given a ball $B$ and $\gamma >0$, we shall often denote by $\gamma B$ ($B/\gamma$) the concentric ball with radius magnified of order $\gamma$ ($1/\gamma$). 
Unless otherwise stated, different balls in the same context will have the same centre. With $\mathcal B \subset \er^{n}$ being a measurable subset with positive measure $0<|\mathcal B|<\infty$, and with $g \colon \mathcal B \to \er^{k}$, $k\geq 1$, being a locally integrable map, we shall denote by  $$
   (g)_{\mathcal B} \equiv \mean{\mathcal B}  g(x) \, dx  := \frac{1}{|\mathcal B|}\int_{\mathcal B}  g(x) \, dx
$$
its integral average. Whenever $\mathcal B \subset \Omega$ being non-empty, we shall also denote, as usual
$$
[a]_{0,\alpha}\equiv  [a]_{0,\alpha; \mathcal B} := \sup_{x,y \in \mathcal B, x \not= y} \, \frac{|a(x)-a(y)|}{|x-y|^\alpha}\,,$$
and $[a]_{0,\alpha}\equiv  [a]_{0,\alpha; \Omega}$. 
We shall  use the auxiliary vector fields $V_p,V_q\colon \er^n \to \er^n$ defined by
$$
V_p(z):= |z|^{(p-2)/2}z \qquad \mbox{and}\qquad V_q(z):= |z|^{(q-2)/2}z
$$
whenever $z \in \ern$. As a consequence of \rif{assa}$_3$, it holds that
\eqn{monotonicity}
$$ 
\left\{
\begin{array}{c}
|V_p(z_1)-V_p(z_2)|^2 + a(x) |V_q(z_1)-V_q(z_2)|^2  \leq c \langle A(x,z_1)-A(x, z_2), z_1-z_2\rangle\\ [8 pt]
H(x, z)  \leq c\langle A(x,z), z\rangle 
\end{array}
\right.
$$
whenever $z,z_1, z_2 \in \er^n$, $x\in \Omega$, and  
with $c \equiv c (n,\nu, p,q)$. 
where the implied constants still depend on $n\in \en$ and $t \in \{p,q\}$ (see \cite{CM1} and related references). Let us recall some basic terminology about generalized Orlicz-Sobolev spaces, i.e., Sobolev spaces defined by the fact that the distributional derivatives lie in a suitable Orlicz-Musielak space. We are mainly interested in spaces related to the Young type function defined in \rif{defiH} and its constant coefficients variants (see for instance \rif{72} below). 
These are defined by
\eqn{defi-sob}
$$
W^{1, H}(\Omega):= \left\{ u \in W^{1,1}(\Omega)\, \colon \, H(\cdot, Du)\in L^1(\Omega) \right\} \;,
$$
with the local variant being defined in the obvious way and $W^{1, H}_0(\Omega)=W^{1, H}(\Omega)\cap W^{1,p}_0(\Omega)$. For more details we refer to \cite{BCM3} and related references. Next to equations as in \rif{eq} and \rif{eq-hom}, we shall consider boundary value problems involving operators with constant coefficients of the type
\begin{flalign}\label{hommi}
\begin{cases}
\ -\diver A_0(Dv)=0 \ \ \mathrm{in} \ B\\
\ v\in w+W^{1,H_{0}}_{0}(B)\;,
\end{cases}
\end{flalign}
where $B\subset \er^n$ is a ball, $ w\in W^{1,H_{0}}(B)$ and 
\eqn{72}
$$
H_{0}(z):=\snr{z}^{p}+a_{0}\snr{z}^{q}\;,
$$
for some $a_0\geq 0$. 
We consider the following assumptions on $A_0(\cdot)$, that are parallel to those in \rif{assa} (and actually coincide with \rif{assa} when $a(x)\equiv a_0$)) and indeed follow a notation similar to the one in \rif{assa}:
\begin{flalign}\label{73}
\begin{cases}
\ A_0\in C( \mathbb{R}^{n},\mathbb{R}^{n})\cap C^{1}(\mathbb{R}^{n}\setminus\{0\},\mathbb{R}^{n})\\
\ \snr{A_{0}(z)}+\snr{\partial A_{0}(z)}\snr{z}\le L(\snr{z}^{p-1}+a_{0}\snr{z}^{q-1})\\
\ \nu (\snr{z}^{p-2}+a_{0}\snr{z}^{q-2})\snr{\xi}^{2}\le \langle \partial A_{0}(z)\xi,\xi \rangle\;.
\end{cases}
\end{flalign}
Solvability of \rif{hommi} follows using monotonicity methods as described in \cite{CM3}. We then have the following result obtained in \cite[Section 5]{CM3}. We only mention that in the next statement no restriction occurs on the ratio $q/p$; the proof is exactly the same as the one presented in \cite{CM3}. 
\begin{theorem}\label{T4}
Under assumptions \trif{73} with $1<p<q$, let $v\in W^{1,H_{0}}(B)$ be the unique distributional solution to \eqref{hommi} with $w \in W^{1,H_{0}}(B)$. Then
\begin{flalign*}
H_{0}(Dw)\in L^{\gamma}(B)\Rightarrow H_{0}(Dv) \in L^{\gamma}(B) \qquad \mathrm{holds \ for \ every \ }\gamma >1\;.
\end{flalign*}
Moreover, for every $\gamma>1$, there exists a constant $c\equiv c(n,p,q,\nu,L,\gamma)$, which is a non decreasing function of $\snr{B}$ and, in particular, it is independent of $a_{0}$, such that the following inequality holds:
\begin{flalign}\label{74}
\mean{B}[H_{0}(Dv)]^{\gamma} \, dx \le c\mean{B}[H_{0}(Dw)]^{\gamma} \, dx\;.
\end{flalign}
\end{theorem}
The following approximation property extends the ones already exploited in \cite{CM1, ELM}. 
\begin{lemma}\label{nolav}
Under assumptions \trif{ass1} and \eqref{ass4}, if $v\in W^{1,H}(\Omega)$ is such that $H(\cdot,Dv)\in L^{1+\delta}_{\mathrm{loc}}(\Omega)$ for some $\delta\ge0$, then there exists a sequence of smooth functions $\{v_{k}\}_{k\in \N}\subset W^{1,\infty}_{\mathrm{loc}}(\Omega)$ such that, for all $B\Subset \Omega$ with radius $r$, $r\le 1$, there holds that
$$
v_{k}\to v \ \mathrm{strongly \ in \ }W^{1,p(1+\delta)}_{\mathrm{loc}}(\Omega) \quad \mathrm{and}\quad \int_{B}[H(x,Dv_{k})]^{1+\delta} \, dx\to\int_{B}[H(x,Dv)]^{1+\delta} \, dx\;. 
$$
\end{lemma}
\begin{proof}
Let us define
$H_{\delta}(x,z):=\left(\snr{z}^{p}+a(x)\snr{z}^{q}\right)^{1+\delta}\approx \snr{z}^{p(1+\delta)}+[a(x)]^{1+\delta}\snr{z}^{q(1+\delta)}$. It is sufficient to show that
$$
v_{k}\to v \ \mathrm{strongly \ in \ }W^{1,p(1+\delta)}_{\mathrm{loc}}(\Omega) \quad \mathrm{and}\quad \int_{B}H_{\delta}(x,Dv_{k}) \, dx\to\int_{B}H_{\delta}(x,Dv) \, dx\;. 
$$
This is now a consequence of the arguments in \cite{CM1, ELM} since the function $x \mapsto [a(x)]^{1+\delta}$ is still $C^{0,\alpha}$-regular and the newly defined integrand $H_{\delta}(\cdot)$ satisfies the conditions detailed in \cite[Section 4]{CM1}, with $p$ and $q$ replaced by $p(1+\delta)$ and $q(1+\delta)$, respectively.  
\end{proof}
The next lemma has been proved in \cite[Proposition 3.1]{CM3} assuming \rif{ass3}, but for its proof the bound in \rif{ass4} is actually sufficient. 
\begin{lemma}\label{testing-prop} Under assumptions \trif{ass1} and \trif{ass4}, let $B\Subset \Omega$ be a ball and let $S\colon B\to \er^n$ be a measurable vector field such that $H(x, S)\in L^1(B)$ and which is a  
distributional solution to the equation $
-\divo\, T(x, S)=0$ in $B$. Here we assume that the vector field $T\colon B\times \er^n\to \er^n$ satisfies the growth conditions
$$
|T(x,z)| \lesssim |z|^{p-1}+a(x)|z|^{q-1}
$$
for every $x\in B$ and $z\in \er^n$. Then every $\varphi \in W^{1,1}_{0}(B)$ such that $H(x, D\varphi)\in L^1(B)$ satisfies
$$
\int_{B} \langle T(x,S), D\varphi \rangle \, dx =0 \;.
$$
\end{lemma}
We conclude with a classical iteration lemma (see \cite[Chapter 6]{G} for a proof). 
\begin{lemma}\label{iter}B
Let $h\colon [\varrho_{0},\varrho_{1}]\to \mathbb{R}$ be a nonnegative and bounded function, and let $\theta \in (0,1)$ and $A,B\ge 0$, $\gamma_{1},\gamma_{2}\ge 0$ be numbers. Assume that
\begin{flalign*}
h(t)\le \theta h(s)+\frac{A}{(s-t)^{\gamma_{1}}}+\frac{B}{(s-t)^{\gamma_{2}}}
\end{flalign*}
holds for $\varrho_{0}\le t<s\le \varrho_{1}$. Then the following inequality holds with $c\equiv c(\theta,\gamma_{1},\gamma_{2})$:
\begin{flalign*}
h(\varrho_{0})\le \frac{cA}{(\varrho_{1}-\varrho_{0})^{\gamma_{1}}}+\frac{cB}{(\varrho_{1}-\varrho_{0})^{\gamma_{2}}}\;.
\end{flalign*}
\end{lemma}
\section{Higher integrability estimates}\label{higher-sec-1} In this section we fix a ball $B$, with radius $r>0$, such that $B\Subset \Omega$ and $r\leq 1$, and we provide a few existence results and regularity estimates for solutions $w\in W^{1,H}(B)$ to Dirichlet boundary value problems of the type
\begin{flalign}\label{p0}
\begin{cases}
\ -\diver A(x,Dw)=0 \ \ \mathrm{in} \ B\\
\ w\in w_0+W^{1,H}_{0}(B)\,,
\end{cases}
\end{flalign}
where $w_0\in W^{1,H}(B)$ is given boundary datum such that
\eqn{maggiore-inizia}
$$
H(x, Dw_0) \in L^{1+\delta}(B) \qquad \mbox{for some}\ \delta>0
$$
and 
\eqn{elleuno}
$$
\|H(\cdot, Dw_0)\|_{L^{1}(B)}\leq L_1 \;,
$$
for some finite constant $L_1\geq 0$. Needless to say, in \rif{p0} and for the rest of the section, the vector $A(\cdot)$ satisfies the assumptions of Theorem \ref{cz}, that is \rif{ass1}, \rif{ass4} and \rif{assa}. We start introducing the perturbed vector fields $A_{k}\colon \Omega \times \er^n \to \er^n$ as
\eqn{perty}
$$
A_{k}(x,z):= A(x,z)+\eps_{k}\snr{z}^{s-2}z\;,
$$ 
where $s>q$ and $\{\eps_{k}\}$ is a sequence of positive numbers such that $\eps_{k}\to 0$. Moreover, with $B\Subset \Omega$, let us consider a sequence $\{\tilde{w}_{k}\}\subset W^{1,\infty}(B)$ such that 
\eqn{conve2}
$$
\int_{B}H(x,D\tilde{w}_{k}) \, dx\to\int_{B}H(x,Dw_0) \, dx\quad \mbox{and} \quad 
\int_{B}[H(x,D\tilde{w}_{k})]^{1+\delta} \, dx\to\int_{B}[H(x,Dw_0)]^{1+\delta} \, dx\;.
$$
The existence of such a sequence is ensured by Lemma \ref{nolav} and by \rif{maggiore-inizia}. Beside the functions $\tilde{w}_{k}$, we consider another sequence $\{w_k\}\subset W^{1,s}(B)$ in such a way that for every $k \in \en$, $w_{k}$ is the unique solution to the Dirichlet problem
\begin{flalign}\label{10}
\begin{cases}
\ -\diver A_{k}(x,Dw_{k})=0 \ \ \mathrm{in} \ B\\
\ w_{k} \in \tilde{w}_{k}+W^{1,s}_{0}(B)\;.
\end{cases}
\end{flalign}
We start with a first higher integrability result extending those originally found in \cite{CM1, CM2}. 
\begin{lemma}\label{unintgeh}
Under assumptions \trif{ass1}, \trif{ass4} and \trif{assa}, assume also that 
\eqn{uniform-bound}
$$\sup_{k \in \N}\, \nr{H(\cdot, Dw_{k})}_{L^{1}(B)}\le \tilde{c}$$ holds for a positive constant $\tilde{c}$. 
Then there exist two positive constants $\delta_1>0$ and $c$, both depending on $\tilde{c},n,p,q,\nu,L,s,[a]_{0,\alpha}$ and $\alpha$, but otherwise independent of $k$, such that
\begin{flalign}\label{9}
\left(\mean{B_{\varrho}}[H(x,Dw_{k})+\eps_{k}\snr{Dw_{k}}^{s}]^{1+\delta_1} \, dx\right)^{\frac{1}{1+\delta_1}}\le c\mean{B_{2\varrho}}H(x,Dw_{k})+\eps_{k}\snr{Dw_{k}}^{s} \, dx\;.
\end{flalign}
holds for every ball $B_{2\varrho}\subset B$ and every $k\in \en$. 
\end{lemma}
\begin{proof} The proof combines a suitable Caccioppoli type estimate with the  Sobolev-Poincar\'e type inequalities obtained in \cite{CM1, CM2, OK2}. 
Let $B_{2\varrho}\Subset B$ be a ball and $  \eta \in C^{1}_{c}(B_{2\varrho})$ be so that $\chi_{B_{\varrho}}\le   \eta \le \chi_{B_{2\varrho}}$ and $\snr{D  \eta}\le 4\varrho^{-1}$. We test the weak formulation of \eqref{10} against $\varphi=  \eta^{s}(w_{k}-(w_{k})_{B_{2\varrho}})$, which is admissible since $w_{k}\in W^{1,s}(B)$, ($s>q$) for all $k \in \N$. By means of $\eqref{monotonicity}$, Young inequality, and recalling that 
$$s<\frac{q(s-1)}{q-1}< \frac{p(s-1)}{p-1}\;,$$ we obtain, for $\eps \in (0,1)$
\begin{flalign*}
& \int_{B_{2\varrho}}  \eta^{s}\left[H(x,Dw_{k})+\eps_{k}\snr{Dw_{k}}^{s}\right] \, dx \\ & \quad  \leq c\int_{B_{2\varrho}}  \eta^{s-1}\left(\snr{Dw_{k}}^{p-1}+a(x)\snr{Dw_{k}}^{q-1}+\eps_{k}\snr{Dw_{k}}^{s-1} \right)\left |\frac{w_{k}-(w_{k})_{B_{2\varrho}}}{\varrho} \right | \, dx\\
& \quad  \le  \varepsilon \int_{B_{2\varrho}}  \eta^{s}\left[H(x,Dw_{k})+\eps_{k}\snr{Dw_{k}}^{s}\right] \, dx\\
&\qquad +c_\varepsilon\int_{B_{2\varrho}}H\left(x,\frac{w_{k}-(w_{k})_{B_{2\varrho}}}{\varrho}\right)+\eps_{k}\left |\frac{w_{k}-(w_{k})_{B_{2\varrho}}}{\varrho} \right |^{s} \, dx\;.
\end{flalign*}
Choosing $\varepsilon$ small enough and reabsorbing terms, we can conclude that
\begin{flalign}\label{11}
\mean{B_{\varrho}}H(x,Dw_{k})+\eps_{k}\snr{Dw_{k}}^{s} \, dx\le c\mean{B_{2\varrho}}H\left(x,\frac{w_{k}-(w_{k})_{B_{2\varrho}}}{\varrho}\right)+\eps_{k}\left |\frac{w_{k}-(w_{k})_{B_{2\varrho}}}{\varrho} \right |^{s} \, dx
\end{flalign}
holds for $c\equiv c(n,p,q,\nu,L,s)$. We can then use the intrinsic Sobolev-Poincar\'e's inequality developed in \cite{CM1, OK1}, to get 
\begin{flalign}\label{12}
\mean{B_{2\varrho}}H\left(x,\frac{w_{k}-(w_{k})_{B_{2\varrho}}}{\varrho}\right) \, dx \le c\left(\mean{B_{2\varrho}}[H(x,Dw_{k})]^{d} \, dx\right)^{1/d}\;,
\end{flalign}
with $c\equiv c(n,p,q,\alpha,[a]_{0,\alpha},\tilde{c})\geq 1 $ and $d\equiv d(n,p,q,\alpha,[a]_{0,\alpha})< 1 $. Moreover, defining $s_{*}=\max\left\{1,\frac{ns}{n+s}\right\}$, we obtain
\begin{flalign}\label{13}
\mean{B_{2\varrho}}\eps_{k}\left |\frac{w_{k}-(w_{k})_{B_{2\varrho}}}{\varrho} \right |^{s} \, dx\le c\left(\mean{B_{2\varrho}}(\eps_{k}\snr{Dw_{k}}^{s})^{s_{*}/s} \, dx\right)^{s/s_{*}}\;,
\end{flalign}
with $c\equiv c(n,s)$. Let $\tilde{d}:=\max\{s_{*}/s,d\}<1$ and combine \eqref{12} and \eqref{13} with H\"older's inequality to obtain
\begin{flalign}\label{14}
\mean{B_{2\varrho}}H\left(x,\frac{w_{k}-(w_{k})_{B_{2\varrho}}}{\varrho}\right)+\eps_{k}\left |\frac{w_{k}-(w_{k})_{B_{2\varrho}}}{\varrho} \right |^{s}\le c\left(\mean{B_{2\varrho}}[H(x,Dw_{k})+\eps_{k}\snr{Dw_{k}}^{s}]^{\tilde{d}} \, dx \right)^{1/\tilde{d}}\;,
\end{flalign}
where $c\equiv c(n,p,q,\nu,L,s,[a]_{0,\alpha},\alpha,\tilde{c})$ and $\tilde{d}=\tilde{d}(n,p,q,s)$. From \eqref{11} and \eqref{14} we obtain
\eqn{getting}
$$
\mean{B_{\varrho}}H(x,Dw_{k})+\eps_{k}\snr{Dw_{k}}^{s} \, dx \le c\left(\mean{B_{2\varrho}}[H(x,Dw_{k})+\eps_{k}\snr{Dw_{k}}^{s}]^{\tilde{d}} \, dx \right)^{1/\tilde{d}}\;,
$$
so we can apply a standard variant of Gehring's lemma to conclude with the statement.
\end{proof}
We then obtain a global version of the above result. \begin{lemma} \label{unbougeh}
Under the assumptions \trif{ass1}, \trif{ass4}, \trif{assa}, \trif{maggiore-inizia}, \trif{elleuno}, \trif{uniform-bound}, assume also that 
\eqn{uniform-bound-g}
$$\sup_{k \in \N}\, \nr{H(\cdot, D\tw_{k})}_{L^{1}(B)}\le \bar{c}$$ holds for a positive constant $\bar{c}$. Then there exist a positive exponent $\delta_2\leq \delta_1$, and constant $c$, both depending only on $ \tilde{c},\bar{c}, n,p,q,\alpha,\nu,L,[a]_{0,\alpha}$, such that 
\begin{flalign*}
\left(\mean{B}[H(x,Dw_{k})+\eps_{k}\snr{Dw_{k}}^{s}]^{1+\sigma} \, dx\right)^{\frac{1}{1+\sigma}}\le c\left(\mean{B}[H(x,D\tilde{w}_{k})+\eps_{k}\snr{D\tilde{w}_{k}}^{s}]^{1+\sigma} \, dx\right)^{\frac{1}{1+\sigma}}
\end{flalign*}
holds whenever $\sigma \in [0, \delta_2)$, for every $k\in \en$. 
\end{lemma}
\begin{proof}
We first test \eqref{10} against $\phi:=w_{k}-\tilde{w}_{k}$ to obtain by $\eqref{monotonicity}$ (arguing for instance as in (see again \cite[Theorem 3.1]{CM3}))
\begin{flalign}\label{15}
\int_{B}H(x,Dw_{k})+\eps_{k}\snr{Dw_{k}}^{s} \, dx \le c\int_{B}H(x,D\tilde{w}_{k})+\eps_{k}\snr{D\tilde{w}_{k}}^{s} \, dx\;,
\end{flalign}
with $c\equiv c(n,p,q,\nu,L,s)$. Now, for $x_{0}\in B$, let $B_{2\varrho}(x_{0})\subset \mathbb{R}^{n}$ be any ball such that $\snr{B_{2\varrho}(x_{0})\setminus B}>\snr{B_{2\varrho}(x_{0})}/10$ and let $\varphi:= \eta^{s}(w_{k}-\tilde{w}_{k})$, with $  \eta\in C^{1}_{c}(B_{2\varrho})$ so that $\chi_{B_{\varrho}}\le   \eta \le \chi_{B_{2\varrho}}$ and $\snr{D  \eta}\le 4\varrho^{-1}$. Notice that $\varphi$ is admissible for testing and $\supp \varphi \subset B\cap B_{2\varrho}(x_{0})$. From the weak formulation of \eqref{10} and \rif{monotonicity}, we obtain
\begin{flalign}\label{16}
\int_{B\cap B_{2\varrho}(x_{0})} \eta^{s}[H(x,Dw_{k})+\eps_{k}\snr{Dw_{k}}^{s}] \, dx & \le c\int_{B\cap B_{2\varrho}(x_{0})}  \eta^{s-1}\snr{A_{k}(x,Dw_{k})}\left |\frac{w_{k}-\tilde{w}_{k}}{\varrho} \right | \, dx\nonumber \\
& \quad +c\int_{B\cap B_{2\varrho}(x_{0})} \eta^{s}\snr{A_{k}(x,Dw_{k})}\snr{D\tilde{w}_{k}} \, dx=:\mathrm{(I)}+\mathrm{(II)}\;,
\end{flalign}
with $c\equiv c(n,p,q,\nu,L,s)$. Again Young inequality gives, for any $\eps \in (0,1)$
\begin{flalign*}
\snr{\mathrm{(I)}}\le \varepsilon \int_{B\cap B_{2\varrho}(x_{0})} \eta^{s}[H(x,Dw_{k})+\eps_{k}\snr{Dw_{k}}^{s}] \, dx + c_\varepsilon\int_{B\cap B_{2\varrho}(x_{0})}H\left(x,\frac{w_{k}-\tilde{w}_{k}}{\varrho}\right)+\eps_{k}\left |\frac{w_{k}-\tilde{w}_{k}}{\varrho} \right |^{s} \, dx\;\;,
\end{flalign*}
where $c_{\eps}\equiv c_{\eps}(n,p,q,\nu,L,s, \eps)$ and, similarly, 
\begin{flalign*}
\snr{\mathrm{(II)}}& \le \varepsilon\int_{B\cap B_{2\varrho}(x_{0})} \eta^{s}[H(x,Dw_{k})+\eps_{k}\snr{Dw_{k}}^{s}] \, dx+c_\varepsilon\int_{B\cap B_{2\varrho}(x_{0})}H(x,D\tilde{w}_{k})+\eps_{k}\snr{D\tilde{w}_{k}}^{s} \, dx\;\;,
\end{flalign*}
where $c_\eps\equiv c_\eps(n,p,q,\nu,L,s, \eps)$. Taking $\eps$ small enough and inserting the content of the last two displays in \eqref{16}, we obtain
\begin{flalign*}
\int_{B\cap B_{\varrho}(x_{0})}H(x,Dw_{k})+\eps_{k}\snr{Dw_{k}}^{s} \, dx &\le  c\int_{B\cap B_{2\varrho}(x_{0})}H\left(x,\frac{w_{k}-\tilde{w}_{k}}{\varrho}\right)+\eps_{k}\left |\frac{w_{k}-\tilde{w}_{k}}{\varrho} \right |^{s} \, dx\\
&\quad + c\int_{B\cap B_{2\varrho}(x_{0})}H(x,D\tilde{w}_{k})+\eps_{k}\snr{D\tilde{w}_{k}}^{s} \, dx\;.
\end{flalign*}
Applying Sobolev-Poincar\'e's inequality as follows
\begin{flalign*}
\mean{B\cap B_{2\varrho}(x_{0})}H\left(x,\frac{w_{k}-\tilde{w}_{k}}{\varrho}\right) \, dx  & \le c\left(\mean{B\cap B_{2\varrho}(x_{0})}[H(x,Dw_{k}-D\tilde{w}_{k})]^{d} \, dx\right)^{1/d}\\
& \le c\left(\mean{B\cap B_{2\varrho}(x_{0})}[H(x,Dw_{k})]^{d} \, dx\right)^{1/d} +  c \mean{B\cap B_{2\varrho}(x_{0})} H(x,D\tilde{w}_{k}) \, dx
\;,
\end{flalign*}
with $c, d$ as in \rif{12}
(this holds with the proof given in \cite{CM1, OK2}) and additionally depending on constant $\bar c$ appearing , after standard manipulations as in the proof of Lemma \ref{unintgeh}, we get
\begin{flalign*}
\mean{B\cap B_{\varrho}(x_{0})}H(x,Dw_{k})+\eps_{k}\snr{Dw_{k}}^{s} \, dx & \le c\left(\mean{B\cap B_{2\varrho}(x_{0})}[H(x,Dw_{k})+\eps_{k}\snr{Dw_{k}}^{s}]^{\tilde{d}} \, dx\right)^{1/\tilde{d}}\nonumber \\ &\quad + c\mean{B\cap B_{2\varrho}(x_{0})}H(x,D\tilde{w}_{k})+\eps_{k}\snr{D\tilde{w}_{k}}^{s} \, dx
\end{flalign*}
with $c\equiv c(\tilde{c},n,p,q,\alpha,\nu,L,s,[a]_{0,\alpha})$. We next consider the situation when it is $B_{2\varrho}(x_{0}) \Subset B$. In this case we can proceed as for the interior case treated in Lemma \ref{unintgeh} getting \rif{9}, thereby getting \rif{getting}. The two cases can be combined via a standard covering argument. More precisely, upon defining 
\begin{flalign*}
V_{k}(x):=\begin{cases}
\ [H(x,Dw_{k})+\eps_{k}\snr{Dw_{k}}^{s}]^{\tilde{d}} \ \  &B\\
\ 0 \quad &\mathbb{R}^{n}\setminus B
\end{cases}\quad and \quad U_{k}(x):=\begin{cases}
\ H(x,D\tilde{w}_{k})+\eps_{k}\snr{D\tilde{w}_{k}}^{s}  \ \  &B\\
\ 0 \quad &\mathbb{R}^{n}\setminus B
\end{cases}
\end{flalign*}
we easily get
\begin{flalign*}
\mean{B_{\varrho}(x_{0})}[V_{k}(x)]^{1/\tilde{d}} \, dx\le c\left \{ \left(\mean{B_{2\varrho}(x_{0})}V_{k}(x) \, dx\right)^{1/\tilde{d}}+\mean{B_{2\varrho}(x_{0})}U_{k}(x) \, dx \right\}\;,
\end{flalign*}
with $c= c(\tilde{c},n,p,q,\nu,L,s,[a]_{0,\alpha},\alpha)$ and $0<\tilde{d}<1$. At this point the conclusion follows once again by the usual variant of Gehring's lemma.
\end{proof}
We proceed with a fractional differentiability result, following a strategy that has been initially implemented in \cite{CM1}. 
\begin{theorem}\label{T1}
Under assumption \trif{ass1}, \trif{ass4}, \trif{assa}, \trif{maggiore-inizia} and \trif{elleuno}, there exists a unique solution $w\in w_{0}+W^{1,H}_0(B)$ to \trif{p0} and it satisfies
\begin{flalign}\label{39}
Dw \in L_{\mathrm{loc}}^{\frac{np}{n-2\beta}}(B)\cap W^{\min\{2\beta/p,\beta\},p}_{\mathrm{loc}}(B)
\end{flalign}
for every $\beta<\alpha$. In particular, it follows 
\begin{flalign}\label{41}
Dw \in L^{2q-p}_{\mathrm{loc}}(B)\;.\end{flalign}
The energy estimate
\begin{flalign}\label{40}
\int_{B}H(x,Dw) \, dx \le c_{1}\int_{B}H(x,Dw_{0}) \, dx
\end{flalign}
holds for a constant $c_{1}\equiv c_{1}(n,\nu,L,p,q)$, while the global higher integrability estimate
\begin{flalign}\label{40bis}
 \int_{B}[H(x,Dw)]^{1+\sigma} \, dx \le c_{2}\int_{B}[H(x,Dw_{0})]^{1+\sigma} \, dx
\end{flalign}
holds for positive constants $c_{2}, \sigma\equiv c_{2}, \sigma(n,p,q,\alpha,\nu,L,[a]_{0,\alpha}, L_1)$ and it is $\sigma < \delta$ . Moreover, in \trif{40bis} the exponent $\sigma$ can be replaced by any smaller positive number. Finally, the solution $w$ can be obtained as the limit of solutions $\{w_k\}$ to problems \trif{10} in the sense that, up to relabelled subsequences, it holds that
\eqn{convergenze}
$$
w_k \rightharpoonup w \quad \mbox{in $W^{1,p(1+\sigma)}(B) $} \qquad  \mbox{and } \qquad 
w_k \to w \quad \mbox{in} \  W^{1,\frac{np}{n-2\beta}}_{\loc} (B)\;,
$$
for every $\beta < \alpha$. In particular, we can choose $\beta$ such that 
$$p(1+\sigma)< 2q-p < \frac{np}{n-2\beta}\;.$$ 
\end{theorem}
\begin{proof}
\emph{Step 1: Approximation}.
We preliminary recall that $\tilde{w}_k\in W^{1,\infty}(B)$ for every integer $k$; for a number $\sigma>0$ which is yet to be defined, the quantity 
\begin{flalign}\label{epssi}
\eps_{k}:=\left(k+\nr{D\tilde{w}_{k}}^{3(2q-p)}_{L^{2q-p}(B)}+\nr{D\tilde{w}_{k}}_{L^{(1+\sigma)(2q-p)}(B)}^{3(1+\sigma)(2q-p)}\right)^{-1} 
\end{flalign}
is well-defined and finite. Moreover, observe that the convergence 
$$
\lim_{k}\, \eps_{k}\int_{B}\snr{D\tilde{w}_{k}}^{2q-p} \, dx=0
$$
happens to be uniform with respect to the choice of $\sigma\geq 0$. We take this choice of $\{\eps_k\}$ in \rif{perty}, where we also choose $s := 2q-p$. We later specify the actual value of $\sigma$. 
The weak form of \eqref{10} is
\begin{flalign}\label{wfk}
\int_{B}A_{k}(x,Dw_{k})\cdot D\varphi \, dx =0 \ \ \mathrm{for \ all \ }\varphi \in W^{1,2q-p}_{0}(B)
\end{flalign}
and \rif{15} becomes
\eqn{24}
$$
\int_{B}H(x,Dw_{k})+\eps_{k}\snr{Dw_{k}}^{2q-p} \, dx
\le c\int_{B}H(x,D\tilde{w}_{k})+\eps_{k}\snr{D\tilde{w}_{k}}^{2q-p} \, dx
$$
with $c\equiv c(n,\nu,L,p,q)$, so that, recalling \rif{epssi} we also have, for $k$ large enough, we have
\eqn{23bis} 
$$
\int_{B}H(x,D\tilde w_{k}) \, dx \stackleq{conve2} 2\int_{B}H(x,Dw_{0}) \, dx \leq 2 L_1:=\bar c\;,
$$
and 
\begin{flalign}\label{23}
\int_{B}H(x,Dw_{k}) \, dx \stackleq{conve2} 2\int_{B}H(x,Dw_{0}) \, dx +c \stackleq{elleuno} c (L_1+1):=\tilde c\;,
\end{flalign}
again for $c\equiv c(n,\nu,L,p,q)$. This now fixes the choice of the numbers $\tilde c$ and $\bar{c}$ appearing in \rif{uniform-bound} and \rif{uniform-bound-g}, respectively, and therefore this ultimately reflects in the value of the two higher integrability exponents $\delta_2\leq \delta_1$,  appearing in Lemmas \ref{unintgeh} and \ref{unbougeh}, respectively, that are independent of the number $\sigma>0$ introduced in \rif{epssi}. Needless to say, and with no loss of generality, we shall consider always the indexes $k$ large enough for which \rif{23bis}-\rif{23} hold. We fix a positive $\sigma$ such that 
\eqn{sceltasigma}
$$
\sigma< \frac{\delta_2}{2}\,, \quad p(1+\sigma) < 2q-p\quad \mbox{and} \quad \sigma < \frac{2\alpha}{n}\;,
$$
and this finally fixes the choice in \rif{epssi}; observe that $\sigma$ exhibits the following dependence:
\eqn{sigma-dep}
$$\sigma \equiv \sigma\left(n,p,q,\alpha,\nu,L,[a]_{0,\alpha},L_1\right)\;.$$ 
Similarly, using this time Lemma \ref{unbougeh}, and again \rif{epssi} we estimate
\begin{flalign}\label{33}
\int_{B}\snr{Dw_{k}}^{p(1+\sigma)} \, dx&\le\int_{B}[H(x,Dw_{k})]^{1+\sigma} \, dx\le \int_{B}[H(x,Dw_{k})+\eps_{k}\snr{Dw_{k}}^{2q-p}]^{1+\sigma} \, dx\nonumber \\
&\le c\int_{B}[H(x,D\tw_{k})+\eps_{k}\snr{D\tw_{k}}^{2q-p}]^{1+\sigma} \, dx  \leq  c\int_{B}[H(x,Dw_{0})]^{1+\sigma} \, dx +c\;,
\end{flalign}
for $c\equiv c(n,p,q,\alpha,\nu,L,[a]_{0,\alpha}, L_1)$, for $k$ large enough.  
We can therefore assume that, up to passing to not relabelled subsequences, $w_{k}\rightharpoonup w$ in $W^{1,p(1+\sigma)}(B)$ for some $w \in w_{0}+W^{1,p(1+\sigma)}_{0}(B)$. By lower semicontinuity in \eqref{24} and \eqref{33}, and again recalling \rif{epssi}, we find:
\begin{flalign}\label{26}
\int_{B}H(x,Dw) \, dx \le c_{1}\int_{B}H(x,Dw_{0}) \, dx \ \ \mathrm{and}\ \ \int_{B}[H(x,Dw)]^{1+\sigma} \, dx \le c_{2}\int_{B}[H(x,Dw_{0})]^{1+\sigma} \, dx\;,
\end{flalign}
with $c_{1}\equiv c_{1}(n,\nu,L,p,q)$ and $c_{2}\equiv c_{2}(n,p,q,\alpha,\nu,L,[a]_{0,\alpha},L_1)$.\\\\
\emph{Step 2: Fractional Sobolev embedding and interpolation}. Here we modify the arguments of  \cite[Section 5]{CM1}. Let us take a ball $B_{2\varrho}\Subset B$ (not necessarily concentric to $B$); 
the computations made in \cite[Section 5, p. 470]{CM1}, which hold when assuming \rif{ass4} too, give that, for $0<\varrho\le t<s\le 2\varrho$ (and after scaling back in the proof given in \cite{CM1})
\begin{flalign}\label{27}
&\nr{V_{p}(Dw_{k})}^{2}_{L^{\frac{2n}{n-2\beta}}(B_{t})} +[V_{p}(Dw_{k})]^{2}_{W^{\beta,2}(B_{t})}\nonumber \\
&\le \frac{c}{(s-t)^{2\beta}}\nr{Dw_{k}}_{L^{p}(B_{s})}^{p}+\frac{c}{(s-t)^{2\beta}}\left(\nr{a}^{2}_{L^{\infty}(B_{s})}+\varrho^{2\alpha}[a]_{0,\alpha;B_{s}}^2+\eps_{k}\right)\nr{Dw_{k}}^{2q-p}_{L^{2q-p}(B_{s})}\;,
\end{flalign}
for every $\beta< \alpha$ and $k \in \en$, with $c\equiv c(n,p,q,\nu,L,\alpha,\beta)$. In the above display we are using the standard notation for the Gagliardo seminorm
$$
[V_{p}(Dw_{k})]^{2}_{W^{\beta,2}(B_{t})} := \int_{B_t} \int_{B_t} \frac{|V_p(Dw_k(x))-V_p(Dw_k(y))|^{2}}{|x-y|^{n+2\beta}}\, dx \, dy\;.
$$
We refer to \cite{DPV} for basic properties about fractional Sobolev spaces, and to \cite{CM1, ELM} for the specific ones that are relevant here. We now aim at estimating the second term in the right-hand side of \eqref{27}. In particular there holds:
\begin{flalign}\label{0}
\nr{Dw_{k}}^{p}_{L^{\frac{np}{n-2\beta}}(B_{t})}\le \frac{c\nr{Dw_{k}}^{p}_{L^{p}(B_{s})}}{(s-t)^{2\beta}}+\frac{cT_{k}\nr{Dw_{k}}^{2q-p}_{L^{2q-p}(B_{s})}}{(s-t)^{2\beta}}
\end{flalign}
for $\beta \in (0,\alpha)$, where 
\eqn{TK}
$$T_{k}:=\nr{a}^{2}_{L^{\infty}(B_{\varrho})}+\varrho^{2\alpha}[a]_{0,\alpha;B_{\varrho}}^{2}+\eps_{k}$$ and $c\equiv c(n,p,q,\nu,L,\alpha,\beta)$. Notice that this last constant blows-up when $\beta \to \alpha$. Next, we start taking $\beta<\alpha$ such that
\eqn{sotto}
$$
2q-p < \frac{np}{n-2\beta}
$$
which is implied, by virtue of \rif{ass4}, by
$$
\frac{\alpha}{1+2\alpha/n} < \beta< \alpha
$$
which is in particular satisfied by choosing 
\eqn{byvirtue}
$$
\frac{\alpha}{1+\sigma} < \beta< \alpha
$$
in view of the last inequality in \rif{sceltasigma}. Keeping \rif{sceltasigma} and \rif{sotto} in mind, we now look for $\theta \in (0,1)$ such that
\eqn{by22}
$$
\frac{1}{2q-p}=\frac{1-\theta}{p(1+\sigma)}+\frac{(n-2\beta)\theta}{np}
$$
and this gives
$$
\theta=\frac{n(2q-2p-p\sigma)}{[2\beta-(n-2\beta)\sigma](2q-p)}\;.
$$
We notice that, keeping also $\eqref{ass4}$ in mind,
\begin{flalign}\label{19}
\theta (2q-p)<p\Leftrightarrow\frac{n(2q-2p-p\sigma)}{2\beta-(n-2\beta)\sigma}<p \Leftrightarrow 2n(q-p)<2\beta p(1+\sigma)\Leftarrow \alpha<\beta(1+\sigma)\;,
\end{flalign}
and the last inequality is true by virtue of the choice made in \rif{byvirtue}. By the choices made above, and in particular by \rif{by22}, we use the interpolation inequality
\begin{flalign*}
\nr{Dw_{k}}_{L^{2q-p}(B_{s})}\le \nr{Dw_{k}}^{(1-\theta)}_{L^{p(1+\sigma)}(B_{s})}\nr{Dw_{k}}_{L^{\frac{np}{n-2\beta}}(B_{s})}^{\theta}\;.
\end{flalign*}
In \eqref{19}, we saw that $(2q-p)\theta<p$, so we may apply Young's inequality with conjugate exponents
\begin{flalign*}
t_{1}=\frac{p}{\theta(2q-p)}\quad \mathrm{and}\quad t_{2}=\frac{p}{p-(2q-p)\theta}\;,
\end{flalign*}
to obtain, for $\eps \in (0,1)$
\begin{flalign}\label{20}
\frac{T_{k}}{(s-t)^{2\beta}}\nr{Dw_{k}}^{2q-p}_{L^{2q-p}(B_{s})}& \le \frac{T_{k}}{(s-t)^{2\beta}}\nr{Dw_{k}}^{(1-\theta)(2q-p)}_{L^{p(1+\sigma)}(B_{s})}\nr{Dw_{k}}^{\theta(2q-p)}_{L^{\frac{np}{n-2\beta}}(B_{s})}\nonumber\\
& \le \varepsilon \nr{Dw_{k}}^{p}_{L^{\frac{np}{n-2\beta}}(B_{s})}+c_\eps\left[\frac{T_{k}}{(s-t)^{2\beta}}\nr{Dw_{k}}_{L^{p(1+\sigma)}(B_{s})}^{(1-\theta)(2q-p)}\right]^{\frac{p}{p-(2q-p)\theta}}.
\end{flalign}
Merging \eqref{20} with \eqref{0} and choosing 
\eqn{depconstant}
$$\eps\equiv \eps(n,p,q,\nu,L,\alpha,\beta, \sigma)\equiv  \eps\left(n,p,q,\alpha,\nu,L,[a]_{0,\alpha},\beta,L_1\right)$$ small enough, we conclude that
\begin{flalign*}
\nr{Dw_{k}}^{p}_{L^{\frac{np}{n-2\beta}}(B_{t})}&\le \frac{1}{2}\nr{Dw_{k}}_{L^{\frac{np}{n-2\beta}}(B_{s})}^{p}+\frac{c}{(s-t)^{2\beta}}\nr{Dw_{k}}^{p}_{L^{p}(B_{s})}+c\left[\frac{T_{k}}{(s-t)^{2\beta}}\nr{Dw_{k}}_{L^{p(1+\sigma)}(B_{s})}^{(1-\theta)(2q-p)}\right]^{\frac{p}{p-(2q-p)\theta}}\;,
\end{flalign*}
where $c$ exhibits the same dependence on the constants appearing in \rif{depconstant} as an effect we also determining the value of $c_{\eps}$ in \rif{20}.  
From Lemma \ref{iter}, we have
\begin{flalign}\label{28}
\nr{Dw_{k}}^{p}_{L^{\frac{np}{n-2\beta}}(B_{\varrho})}\le \frac{c}{\varrho^{2\beta}}\nr{Dw_{k}}^{p}_{L^{p}(B_{2\varrho})}+\left[\frac{T_{k}}{\varrho^{2\beta}}\nr{Dw_{k}}_{L^{p(1+\sigma)}(B_{2\varrho})}^{(1-\theta)(2q-p)}\right]^{\frac{p}{p-(2q-p)\theta}}  ,
\end{flalign}
again with $c$ depending as in \rif{depconstant} and for every $k \in \en$. Recalling \rif{sotto} so, by \eqref{28}, H\"older's inequality, \eqref{23} and \eqref{33}, and recalling that $\varrho\leq 1$, we can conclude that
\begin{flalign*}
\nr{Dw_{k}}_{L^{2q-p}(B_{\varrho})}+ \nr{Dw_{k}}_{L^{\frac{np}{n-2\beta}}(B_{\varrho})}&\le \frac{c}{\varrho^{2\beta/p}}\nr{Dw_{k}}_{L^{p}(B_{2\varrho})}+c\left[\frac{T_{k}}{\varrho^{2\beta}}\nr{Dw_{k}}_{L^{p(1+\sigma)}(B_{2\varrho})}^{(1-\theta)(2q-p)}\right]^{\frac{1}{p-(2q-p)\theta}}  \nonumber\\
&\le  \frac{c }{\varrho^{2\beta/p}}\left[\nr{H(\cdot,Dw_{0})}_{L^{1}(B)}^{1/p}+1\right]\notag \\ & \qquad +c\left[\frac{T_{k}}{\varrho^{2\beta}}\left(\nr{H(\cdot,Dw_{0})}_{L^{1+\sigma}(B)}^{1/p}+1\right)^{(1-\theta)(2q-p)}\right]^{\frac{1}{p-(2q-p)\theta}}\;,
\end{flalign*}
holds with $c\equiv c(n,p,q,\alpha,\nu,L,[a]_{0,\alpha},L_1)$ and for sufficiently large $k$. All in all, we deduce that \begin{flalign}\label{32}
\nr{Dw_{k}}_{L^{2q-p}(B_{\varrho})}+ \nr{Dw_{k}}_{L^{\frac{np}{n-2\beta}}(B_{\varrho})}\le c\varrho^{-\lambda}\;,
\end{flalign}
holds for $c,\lambda\equiv c,\lambda(n,p,q,\nu,L,\nr{a}_{\infty},[a]_{0,\alpha},\alpha,\beta,L_1)$, which is independent of $k$. This last estimate used together with \rif{27} and a standard covering argument (keep again the proof of \cite[Theorem 3.1]{CM3} in mind), allows to get the new bound\
\begin{flalign}\label{37}
\nr{Dw_{k}}_{W^{\min\{2\beta/p,\beta\},p}(U)}+\nr{V_{p}(Dw_{k})}_{W^{\beta,2}(U)}\le c\end{flalign}
for any open subset $U\Subset B$ and any $\beta < \alpha$, with $c$ depending as in \rif{32}, and additionally on $\dist(U,\partial B)$. 
By \rif {33} and \rif{37} we can use the standard compact embedding theorems of fractional Sobolev spaces and again a standard diagonal argument allows to conclude that, up to (not relabelled) subsequence, it holds that $Dw_{k}\to Dw$ strongly in $L^{2q-p}_{\mathrm{loc}}(B)$ and a.e. (this completely proves \rif{convergenze}). This allows to let $k\to \infty$ in \eqref{wfk}, obtaining that $w$ is a distributional solution to \eqref{p0}. Using lower semicontinuity in \rif{37} then yields \rif{39}-\rif{41}, while \rif{40}-\rif{40bis} are a consequence of \rif{26}. Finally, let $\tilde{w}$ be another distributional solution to \eqref{p0} such that $\tilde{w} \in W^{1,H}(B)$. By Lemma \ref{testing-prop} it follows that we can use $\varphi=w-\tilde{w}$ as test function in the weak formulation
\begin{flalign*}
\int_{B}(A(x,Dw)-A(x,D\tilde{w}))\cdot(Dw-D\tilde{w}) \, dx=0\;.
\end{flalign*}
At this point, the strict monotonicity \eqref{monotonicity} of the vector field $A(\cdot)$ gives that $\tilde{w}=w$. 
\end{proof}

\section{Another higher integrability estimate}\label{higher-sec-2}
This is in the following:
\begin{theorem}\label{T-higher}
Let $u \in W^{1,H}(\Omega)$ be a solution to the equation \trif{eq}, under the assumptions \trif{ass4} and \trif{assa}-\trif{G}. Assume, moreover, that $H(\cdot, F)\in L^{\gamma}_{\loc}(\Omega)$, for some $\gamma>1$. 
There exists a positive higher integrability exponent $\delta< \gamma-1$, 
depending only on $\data$, such that 
$H(\cdot, Du) \in L^{1+\delta}_{\loc}(\Omega)$; moreover, the reverse H\"older type inequality
\begin{flalign}\label{46tris}
\left(\mean{B_{\varrho}}[H(x, Du)]^{1+\delta} \, dx\right)^{\frac{1}{1+\delta}}  \le c\mean{B_{2\varrho}}H(x,Du) \, dx + c 
\left(\mean{B_{2\varrho}}[H(x, F)]^{1+\delta} \, dx\right)^{\frac{1}{1+\delta}} 
\end{flalign}
holds for every ball $B_{2\varrho}\subset \Omega$, where the constant $c$ depends again only on $\data$. In particular, in the case $F \equiv 0$, for every open subset $\Omega_0 \Subset \Omega$ there exists a constant $c\equiv c(\textnormal{\data}, \dist (\Omega_0, \partial \Omega))\geq 1$ such that
$$
\|H(x, Du)\|_{L^{1+\delta}(\Omega_0)}\leq c \;.
$$
Moreover, in \trif{46tris} the exponent $\delta$ can be replaced by any smaller number. 
\end{theorem}
\begin{proof} The proof is similar to the one of Lemma \ref{unintgeh} and to the one offered in \cite[Theorem 1.1]{CM1}, and we shall report only a brief sketch.   
Let $B_{2\varrho}\Subset \Omega$. We test the weak formulation of \eqref{eq} against $\varphi:= \eta^{q}(u-(u)_{B_{2\varrho}})$, where $\eta \in C^{1}_{c}(B_{2\varrho})$, $\chi_{B_{\varrho}}\le \eta\le \chi_{B_{2\varrho}}$ and $\snr{D\eta}\le 4/\varrho$. Notice that $\varphi$ is admissible by Lemma \ref{testing-prop}. Using \eqref{assa}-\eqref{G}, the fact that $q<\frac{p(q-1)}{p-1}$, and Young inequality we get
\eqn{91}
$$
\int_{B_{2\varrho}}H(x,Du) \eta^{q} \, dx \le  c\int_{B_{2\varrho}}H\left(x,\frac{u-(u)_{B_{2\varrho}}}{\varrho}\right) \, dx +  c\int_{B_{2\varrho}}H(x,F)\, dx\;, 
$$
with $c\equiv c(n,\nu,L,p,q)$. Now we apply Sobolev-Poincar\'e's inequality to the first term on the right-hand side of \eqref{91} to have
\begin{flalign*}
\mint_{B_{\varrho}}H(x,Du) \, dx 
\le c\left(\mint_{B_{2\varrho}}[H(x,Du)]^{d} \, dx\right)^{1/d}+c\mint_{B_{2\varrho}}H(x,F) \, dx\;,
\end{flalign*}
with $c\equiv c(\data)$ and $d\equiv d(n,p,q) \in (0,1)$.  Now, since $H(\cdot,F)\in L^{\gamma}_{\mathrm{loc}}(\Omega)$ and $\gamma>1$, by a variant of Gehring's Lemma we can conclude with \rif{46tris} for a number $\delta$ as described in the statement of the theorem, and the proof is complete.
\end{proof}

\section{Proof of Theorem \ref{cz}: A conditional reverse H\"older inequality}\label{higher-sec-3}
In this section we start the proof of Theorem \ref{cz}. Our aim is to prove the reverse type inequality in Theorem \ref{T2} below. This extends a similar fact obtained in \cite[Theorem 5.1]{CM1} under the assumption \rif{ass3}; we now replace this by \rif{ass4}. The result we are going to develop here is in fact a technical tool in the forthcoming proof of Theorem \ref{cz} contained in the next section. We start considering the original solution $u$ from Theorem \ref{cz}. By Theorem \ref{T-higher} and a standard covering argument, we know that for every choice of open subset $\Omega_0\Subset \tilde \Omega_0\Subset \Omega$ (with $\dist (\Omega_0 , \partial \tilde \Omega_0)\approx \dist (\tilde \Omega_0 , \partial \Omega)\approx \dist (\Omega_0 , \partial  \Omega)$) as in the statement of Theorem \ref{cz}, there exists a constant $c\geq 1$ such that
\eqn{quellaglobale2}
$$
\|H(\cdot, Du)\|_{L^{1+\delta}(\Omega_0)}\leq  c\left(\textnormal{\data}, \dist (\Omega_0, \partial \Omega), \|H(\cdot, F)\|_{L^\gamma(\tilde \Omega_0)}\right)
$$
holds for some exponent $\delta$ depending only on $\textnormal{\data}$ and which is such that $\delta < \gamma-1$. We then use the setting of Section \ref{higher-sec-1} with $w_0\equiv u$, by considering problems of the type
\begin{flalign}\label{p0-again}
\begin{cases}
\ -\diver A(x,Dw)=0 \ \ \mathrm{in} \ B_{4\varrho}\\
\ w\in u+W^{1,H}_{0}(B_{4\varrho})\,,
\end{cases}
\end{flalign}
where $B_{4\varrho}$ is a ball such that $B_{8\varrho}\subset \Omega_0$ and $8\varrho\leq 1$, with the number $\delta$ coming from \rif{quellaglobale2} as the one fixed in \rif{maggiore-inizia}; the constant $L_1$ appearing in is obviously fixed by $L_1:=\|H(\cdot, Du)\|_{L^{1}(B_{4\varrho})}$. 
Moreover, by \rif{40}, we have
\eqn{superap}
$$
\|H(\cdot, Dw)\|_{L^{1}(B_{4\varrho})}\leq c(n,p,q,\nu, L) \|H(\cdot, Du)\|_{L^{1}(B_{4\varrho})}\;.
$$
In view of this last inequality and by Theorem \ref{T-higher}, this time applied with $F\equiv 0$, we obtain that $Dw \in L^{1+\delta_0}_{\loc}(4B)$ for some $\delta_0\equiv \delta_0(\data)>0$. Moreover, recalling that in Theorem \ref{T1} we can take $\sigma\in (0, \delta]$ as small as we like, and in particular $\sigma \leq \delta_0$, we have the following reverse H\"older type inequality:
\begin{flalign}\label{4600}
 \left(\mean{B_{2\varrho}}[H(x,Dw)]^{1+\sigma} \, dx\right)^{\frac{1}{1+\sigma}} & \le c\mean{B_{4\varrho}}H(x,Dw) \, dx\;,
\end{flalign}
for a constant $c\equiv c (\data)$ and for a final number $\sigma$ depending only on $\data$. The main result of this  section is now
\begin{theorem}\label{T2}
Let $w \in W^{1,H}(B_{4\varrho})$ be a solution to  \trif{p0-again} under the assumptions \eqref{ass4}, \eqref{assa}. Assume that 
\begin{flalign}\label{dega}
\sup_{x \in B_{\varrho}}a(x)\le K[a]_{0,\alpha}\varrho^{\alpha}
\end{flalign}
holds for some $K\ge 1$. Then, for any $\bar{q}<np/(n-2\alpha)$ ($=\infty$ when $\alpha=1$ and $n=2$) there exists a positive constant $c\equiv c(\data, \bar q,K)$, such that the following reverse H\"older type inequality holds:
\begin{flalign}\label{46}
\left(\mean{B_{\varrho}}\snr{Dw}^{\bar{q}} \, dx\right)^{1/\bar{q}} & \le c \left(\mean{B_{4\varrho}}H(x,Dw) \, dx\right)^{1/p}\;.
\end{flalign}
Moreover, let $1< \gamma_1< \gamma_2 \leq 4$, the inequality
\begin{flalign}\label{46bis}
\left(\mean{B_{\gamma_1\varrho}}\snr{Dw}^{2q-p} \, dx\right)^{\frac{1}{2q-p}} & \le c\left(\mean{B_{\gamma_2 \varrho}}H(x,Dw) \, dx\right)^{1/p}
\end{flalign}
holds for a constant $c$ additionally depending on $\gamma_1, \gamma_2$. 
\end{theorem}
\begin{proof}
We shall revisit the arguments in the proof of Theorem \ref{T1}, and we keep the notation introduced there; in particular, we shall retain and \rif{sceltasigma} and \rif{byvirtue} concerning $\beta$ and $\sigma$ (and of course $\sigma$ obeys the smallness conditions enumerated before the statement). Given the reference ball $B_{4\varrho}$ mentioned in the statement of the theorem, all the remaining balls will be concentric unless otherwise stated. We go back to the proof of Theorem \ref{T1}, Step 2,  where we consider this time concentric balls $B_{\varrho}\subset B_{4\varrho}\equiv B$, and the approximate solutions $\{w_k\}$ defined in \rif{10} with the choice $w_0=u$; ultimately, this means we are considering problems \rif{p0-again}. Theorem \ref{T1} implies, in particular, that $Dw_{k}\to Dw$ strongly in $L^{p(1+\sigma)}(B_{2\varrho})$; letting $k\to \infty$ in \eqref{28} and recalling the definition \rif{TK},  we conclude with 
\begin{flalign}\label{44}
\nr{Dw}_{L^{\frac{np}{n-2\beta}}(B_{\varrho})}&\le \frac{c}{\varrho^{2\beta/p}}\nr{Dw}_{L^{p}(B_{2\varrho})}+c\left[\frac{1}{\varrho^{2\beta}}\left(\nr{a}^{2}_{L^{\infty}(B_{2\varrho})}+c\varrho^{2\alpha}[a]^{2}_{0,\alpha;B_{2\varrho}}\right)\nr{Dw}_{L^{p(1+\sigma)}(B_{2\varrho})}^{(1-\theta)(2q-p)}\right]^{\frac{1}{p-(2q-p)\theta}}\;,
\end{flalign}
with $c\equiv c(\data,\beta)$. Notice now that condition \eqref{dega} is stable when the radius increases for nested balls. In fact, \eqref{dega} and the $\alpha$-H\"older continuity of $a(\cdot)$ imply
\begin{flalign}\label{dega54}
\sup_{x\in B_{M\varrho}}a(x) \le (K+3M)\varrho^{\alpha}[a]_{0,\alpha}\;, \qquad \forall \ M \in (1,4)\;,
\end{flalign}
so, with \eqref{dega54}, \eqref{44} becomes (recall $K\geq 1$)
\begin{flalign}\label{49}
\nr{Dw}_{L^{\frac{np}{n-2\beta}}(B_{\varrho})}^{p}&\le \frac{c}{\varrho^{2\beta}}\nr{Dw}_{L^{p}(B_{2\varrho})}^{p}+c\left(\varrho^{2(\alpha-\beta)}[a]_{0,\alpha;B_{2\varrho}}^{2}K^2\nr{Dw}_{L^{p(1+\sigma)}(B_{2\varrho})}^{(1-\theta)(2q-p)}\right)^{\frac{p}{p-(2q-p)\theta}}\;.
\end{flalign}
Define
\begin{flalign*}
b_{1}:=\frac{p}{p-(2q-p)\theta} \quad \mathrm{and}\quad b_{2}:=\frac{(1-\theta)(2q-p)}{(1+\sigma)[p-(2q-p)\theta]}\;\;.
\end{flalign*}
In these terms, after averaging and making a few elementary manipulations, \eqref{49} reads as
\begin{flalign*}
\left(\mean{B_{\varrho}}\snr{Dw}^{\frac{np}{n-2\beta}} \, dx\right)^{\frac{n-2\beta}{np}}& \le c\left(\mean{B_{2\varrho}}\snr{Dw}^{p} \, dx\right)^{1/p}\\
&\quad +c\left([a]_{0,\alpha;B_{2\varrho}}^{2}K^{2}\right)^{b_{1}/p}\varrho^{\frac{2(\alpha-\beta)b_{1}+2\beta +n(b_{2}-1)}{p}}\left(\mean{B_{2\varrho}}\snr{Dw}^{p(1+\sigma)} \, dx\right)^{b_{2}/p}\;.
\end{flalign*}
Let us re-write the last term in the above inequality as follows:
\begin{flalign*}
\left(\mean{B_{2\varrho}}\snr{Dw}^{p(1+\sigma)} \, dx\right)^{b_{2}/p}&=\left(\mean{B_{2\varrho}}\snr{Dw}^{p(1+\sigma)} \, dx\right)^{\frac{b_{2}(1+\sigma)-1}{p(1+\sigma)}}\left(\mean{B_{2\varrho}}\snr{Dw}^{p(1+\sigma)} \, dx\right)^{\frac{1}{p(1+\sigma)}}\\
&= \omega_{n}^{\frac{1-b_{2}(1+\sigma)}{p(1+\sigma)}}\varrho^{-n\frac{b_{2}(1+\sigma)-1}{p(1+\sigma)}}\nr{Dw}_{L^{p(1+\sigma)}(B_{2\varrho})}^{\frac{b_{2}(1+\sigma)-1}{p(1+\sigma)}}\left(\mean{B_{2\varrho}}\snr{Dw}^{p(1+\sigma)} \, dx\right)^{\frac{1}{p(1+\sigma)}}\;\;,
\end{flalign*}
where $\omega_{n}$ is the measure of the $n$-dimensional unit ball. 
Now notice that
$$
\frac{1}{p}\left[2(\alpha-\beta)b_{1}+2\beta+n(b_{2}-1)-nb_{2}+\frac{n}{(1+\sigma)}\right]\ge \frac{1}{p}\left(2\beta-\frac{n\sigma}{1+\sigma}\right)\stackrel{\rif{byvirtue}}{\ge} \frac{2\alpha-n\sigma}{p(1+\sigma)}\stackrel{\rif{sceltasigma}}{>}0\,,
$$
so, merging the content of the previous three displays and using H\"older's inequality, we conclude with
\begin{eqnarray}\label{50}
\left(\mean{B_{\varrho}}\snr{Dw}^{\frac{np}{n-2\beta}} \, dx\right)^{\frac{n-2\beta}{np}}& \leq& c\left\{1+\left([a]_{0,\alpha;B_{2\varrho}}^{2}K^{2}\right)^{b_{1}/p}\nr{Dw}_{L^{p(1+\sigma)}(B_{2\varrho})}^{\frac{b_{2}(1+\sigma)-1}{p(1+\sigma)}}\right\}\left(\mean{B_{2\varrho}}\snr{Dw}^{p(1+\sigma)} \, dx\right)^{\frac{1}{p(1+\sigma)}}\nonumber\\
&\stackleq{superap}&c\left(\mean{B_{2\varrho}}\snr{Dw}^{p(1+\sigma)} \, dx\right)^{\frac{1}{p(1+\sigma)}}\stackleq{4600} c\left(\mean{B_{4\varrho}}H(x,Dw)\, dx\right)^{\frac{1}{p}} \;,
\end{eqnarray}
with $c\equiv c(\data, \beta, K)$. We have therefore proved \rif{46} for the values of $\bar{q}$ such that 
$np/(n-\beta p)\leq  \bar{q} $, where $\beta$ is such that $\alpha/(1+\sigma)< \beta< \alpha$, and with $\sigma$ coming from Theorem \ref{T1} as specified at the beginning of the proof; the same obviously follows using H\"older's inequality. In particular, \rif{46} follows for a suitable choice of $\beta$. As for the \rif{46bis}, this follows by a variant of a standard covering argument starting from the validity of \rif{46}; let us briefly recall it. We can cover $B_{\varrho}$ by a finite number $k\equiv k (n,\gamma)$ of balls $\{B_i\}_{i\leq k}$ touching $B_{\gamma_1\varrho}$ and with radius $\tilde \varrho = (\gamma_2-\gamma_1)\varrho/100$. Obviously, it is $2B_i \Subset B_{\gamma_2\varrho}$ and these balls are not necessarily concentric to the starting ball $B_{\varrho}$. Notice that, for every $i \leq k$, it is
$$
\sup_{x \in B_{i}}a(x)\le\sup_{x \in B_{\gamma\varrho}}a(x)\leq  (K+3\gamma)[a]_{0,\alpha}\varrho^{\alpha}\leq \left[\frac{100(K+3\gamma_2)}{\gamma_2-\gamma_1}\right] [a]_{0,\alpha}\tilde \varrho^\alpha\;.
$$
We can therefore apply \rif{46} to each of the balls $B_i$, thereby getting 
$$
\left(\mean{B_{i}}\snr{Dw}^{\bar{q}} \, dx\right)^{1/\bar{q}}  \le c \left(\mean{2B_{i}}H(x,Dw) \, dx\right)^{1/p}\;,
$$
for a constant $c$ now depending also on $\gamma$. By summing up the above inequalities with respect to $i$ and again increasing the involved constant in a way that depends only on $n,\gamma$, we finally arrive at \rif{46bis} and the proof is complete.  
 \end{proof}
\section{Proof of Theorem \ref{cz}: Exit time arguments and conclusion}

The general scheme of the proof of Theorem \ref{cz} is the same one of \cite[Theorem 1.1]{CM3}. We ask the reader to have \cite{CM3} at hand since we shall essentially indicate the relevant modifications and we shall follow exactly the same steps as described there, reporting them with the same titles. 

\emph{Step 1: Exit time and covering of the level set}. We recall that with open subset $\Omega_0\Subset \tilde \Omega_0\Subset \Omega$ as in the statement of the theorem, we have that \rif{quellaglobale2} holds. 
Then we consider $B_R \subset \Omega_0$, where $R\leq r$ (and $r$ is the small radius appearing in the statement of Theorem \ref{cz} and to be determined at the end of the proof). We proceed with the exit time and covering argument as in the proof of \cite[Theorem 1.1]{CM3}. In particular, this yields the family of balls $\{B_{i}\}\equiv \{B_{\varrho_i}(x_i)\}\equiv \{5\tilde B_{i}\}$ as indicated in \cite[(4.9)-(4.11)]{CM3}. All the balls in question are contained in $B_R$. Before going on, similarly to \rif{dataref}, we set
$$\data_0 \equiv 	\data_0 
\left(n,p,q,\alpha,\ratio,\|a\|_{L^\infty}, [a]_{\alpha}, \|H(\cdot,Du)\|_{L^1(\Omega)},\dist (\Omega_0, \partial \Omega), \|H(\cdot, F)\|_{L^\gamma(\tilde \Omega_0)}, \gamma\right)\;.$$

\emph{Step 2: A first comparison function.} We recover the setting of Sections \ref{higher-sec-1} and \ref{higher-sec-3}. By \rif{quellaglobale2} we use the setting of Section \ref{higher-sec-1} with this choice of $\delta$ in \rif{maggiore-inizia}; in this way we are also ready to use the setting of Section \ref{higher-sec-3} and Theorem \ref{T2}. We are therefore able to apply Theorem \ref{T1} that, in turn, allows to define $w_{i}\in u+W^{1,H}_{0}(4B_{i})$ as the solutions to the Dirichlet problem
\begin{flalign}\label{55}
\begin{cases}
\ -\diver A(x,Dw_{i})=0 \ \ \mathrm{in} \ 4B_{i}\\
\ w_{i}\in u+W^{1,H}_{0}(4B_{i})\;.
\end{cases}
\end{flalign}
where the balls $B_{i}$ are from Step 1. Again, by Theorem \ref{T1} with $w_0\equiv u$ as explained in the previous section, we have 
\begin{flalign}\label{53}
w_{i}\in W^{1,2q-p}_{\mathrm{loc}}(4B_{i}), \quad Dw_{i} \in L^{\frac{np}{n-2\beta}}_{\mathrm{loc}}(4B_{i}) \ \ \mathrm{for \ all \ }\beta<\alpha
\end{flalign}
and
\eqn{54}
$$
\left\{
\begin{array}{c}
\displaystyle 
\mean{4B_{i}}H(x,Dw_{i}) \, dx \le c_{1}\mean{4B_{i}}H(x,Du) \, dx\\ [17 pt]
\displaystyle  \int_{4B_{i}}[H(x,Dw_{i})]^{1+\sigma} \, dx\le c_{2}\int_{4B_{i}}[H(x,Du)]^{1+\sigma} \, dx
\end{array}
\right.
$$
hold for positive constants $c_{1}\equiv c_{1}(n,\nu,L,p,q)$, $c_{2}, \sigma \equiv c_{2}, \sigma(\data)$. As in \cite[(4.17)]{CM3}, we gain that for every $\varepsilon \in (0,1)$ there exists a constant $c_{\varepsilon}$, depending also by $n, \nu, L, p, q$, such that 
\begin{flalign}\label{56}
\mean{4B_{i}}&\snr{V_{p}(Du)-V_{p}(Dw_{i})}^{2}+a(x)\snr{V_{q}(Du)-V_{q}(Dw_{i})}^{2} \, dx \nonumber \\
& \le\varepsilon \mean{4B_{i}}H(x,Du) \, dx +c_{\varepsilon}\mean{4B_{i}}H(x,F) \, dx\;.
\end{flalign}

\emph{Step 3: A second comparison function}. As in \cite{CM3}, we consider a point $x_{i,m}\in \overline{2B_{i}}$ such that 
\eqn{defiafrozen}
$$a(x_{i,m})=\sup_{x\in 2B_{i}}a(x)\;.$$
By \eqref{53} we have in particular that $w_{i}\in W^{1,q}(2B_{i})$ so that, by setting
\eqn{caso-speciale}
$$
H_{i,m}(z):=H(x_{i,m},z)\;,
$$
we have that $w_{i}\in W^{1,H_{i,m}}(2B_{i})$, and, in particular $w_{i}\in W^{1, 2q-p}(2B_{i})$. We can therefore use Theorem \ref{T1} for the special case of $H$-function in \rif{caso-speciale}. This yields 
$v_{i}\in W^{1,H}(2B_{i})\cap W^{1,H_{i,m}}(2B_{i})$ as the unique solution to the Dirichlet problem
\begin{flalign*}
\begin{cases}
\ -\diver A(x_{i,m},Dv_{i})=0 \ \ \mathrm{in} \ 2B_{i}\\
\ v_{i}\in w_{i}+W^{1,H_{i,m}}_{0}(2B_{i})\;,
\end{cases}
\end{flalign*}
for which we get 
\begin{flalign}\label{59}
\mean{2B_{i}}H(x_{i,m},Dv_{i}) \, dx \le c \mean{2B_{i}}H(x_{i,m},Dw_{i}) \, dx\;,
\end{flalign}
with $c\equiv c(n,\nu,L,p,q)$. Notice that the right-hand side of \eqref{59} is finite, sice $Dw_{i}\in L^{q}(2B_{i})$ (by \rif{53}). The weak form of the equations solved by $w_{i}$ and $v_{i}$ respectively can be rewritten as
\begin{flalign}\label{60}
\mean{2B_{i}}&(A(x_{i,m},Dv_{i})-A(x_{i,m},Dw_{i}))\cdot D\varphi \, dx =\mean{2B_{i}}(A(x,Dw_{i})-A(x_{i,m},Dw_{i})) \cdot D\varphi \, dx\;,
\end{flalign}
that holds for every choice of smooth test function $\varphi \in C^{\infty}_{c}(2B_{i})$. As in \cite{CM3}, the function $\varphi:=v_{i}-w_{i}$ is admissible in \eqref{60} and this gives, via \eqref{monotonicity} and $\eqref{assa}_{4}$ (see also \cite[(4.26)]{CM3}),
\begin{flalign}\label{62}
\mean{2B_{i}}&\snr{V_{p}(Dv_{i})-V_{p}(Dw_{i})}^{2}+a(x_{i,m})\snr{V_{q}(Dv_{i})-V_{q}(Dw_{i})}^{2} \, dx \nonumber \\
&\le c\left(\osc_{2B_{i}}a\right)\mean{2B_{i}}\snr{Dw_{i}}^{q-1}\snr{Dv_{i}-Dw_{i}} \, dx=:\mathrm{(I)}\;,
\end{flalign}
where $c\equiv c(n,\nu,L,p,q)$.

\emph{Step 4: Two different phases}. We are now aiming at estimating the term $\mathrm{(I)}$ appearing in the last display. For this, following \cite{CM1, CM3}, we distinguish the two phases, that is
\begin{flalign}\label{pqp}
\inf_{x \in 2B_{i}}a(x)>K[a]_{0,\alpha}\varrho_{i}^{\alpha}\;,
\end{flalign}
which is called, as in \cite{CM1}, the $(p,q)$-phase, and 
\begin{flalign}\label{pp}
\inf_{x \in 2B_{i}}a(x)\le K[a]_{0,\alpha}\varrho_{i}^{\alpha}. 
\end{flalign}
The number $K\ge 4$ is to be determined towards the proof as a quantity depending only on $n,p,q,\nu,L,\gamma$. 

\emph{Step 5: Estimates in the $(p,q)$-phase}. We consider the case \eqref{pqp} holds. Notice that
\begin{flalign*}
\osc_{2B_{i}}a \le 4[a]_{0,\alpha}\varrho_{i}^{\alpha}\le \frac{4a(x)}{K} \quad \mathrm{for \ every \ }x \in 2B_{i}\;,
\end{flalign*}
so that, by Young's inequality yields
\begin{flalign}\label{63}
\mathrm{(I)}& \le\frac{c}{K}\mean{2B_{i}}a(x)\snr{Dw_{i}}^{q-1}\snr{Dv_{i}-Dw_{i}} \, dx \le \frac{c}{K}\mean{2B_{i}}a(x)(\snr{Dw_{i}}^{q}+\snr{Dv_{i}}^{q}) \, dx\;,
\end{flalign}
with $c\equiv c(n,p,q,\nu,L,[a]_{0,\alpha})$. Observing that
\begin{flalign}\label{64}
a(x_{i,m})\le a(x)+4[a]_{0,\alpha}\varrho_{i}^{\alpha}\le a(x)+\frac{4a(x)}{K}\le 2a(x)\;,
\end{flalign}
we estimate
\begin{eqnarray}
\nonumber \mean{2B_{i}}a(x)(\snr{Dw_{i}}^{q}+\snr{Dv_{i}}^{q}) \, dx & \stackleq{defiafrozen} & c \mean{2B_{i}}a(x_{i,m})\left(\snr{Dw_{i}}^{q}+\snr{Dv_{i}}^{q} \right)\, dx\\
\nonumber & \le & c\mean{2B_{i}}H(x_{i,m},Dw_{i}) \, dx + c\mean{2B_{i}}H(x_{i,m},Dv_{i}) \, dx\\
&\stackleq{59}&  c\mean{2B_{i}}H(x_{i,m},Dw_{i}) \, dx \nonumber \\ &
\stackleq{64}  & c\mean{2B_{i}}H(x,Dw_{i}) \, dx \nonumber\\ \nonumber  & \stackleq{54} &c\mean{4B_{i}}H(x,Du) \, dx\;,
\end{eqnarray}
with $c\equiv c(n,p,q,\nu,L,[a]_{0,\alpha})$. Using this with \eqref{63} in \eqref{62} yields
\begin{flalign}\label{65}
\mean{2B_{i}}\snr{V_{p}(Dv_{i})-V_{p}(Dw_{i})}^{2}+a(x_{i,m})\snr{V_{q}(Dv_{i})-V_{q}(Dw_{i})}^{2} \, dx\le \frac{c}{K}\mean{4B_{i}}H(x,Du) \, dx\;,
\end{flalign}
for a constant $c\equiv c(n,p,q,\nu,L,[a]_{0,\alpha})$ which is independent of $K$. We single out the following estimate from the second-last display
\begin{flalign}\label{81pre}
\mean{2B_{i}}H(x_{i,m},Dw_{i}) \, dx \le c\mean{4B_{i}}H(x,Du) \, dx\;, 
\end{flalign}
for $c\equiv c (n,p,q,\nu, L)$. 

\emph{Step 6: Estimates in the $p$-phase}. Here we consider the case in which \eqref{pp} is in force. In this case it is
\begin{flalign}\label{66}
a(x_{i,m})\le 4[a]_{0,\alpha}\varrho_{i}^{\alpha}+\inf_{x \in 2B_{i}}a(x)\le (4+K)[a]_{0,\alpha}\varrho_{i}^{\alpha}\;.
\end{flalign}
Therefore, as described in Step 2, we apply Theorem \ref{T2}, estimate \rif{46bis}, that gives
\begin{flalign}\label{68}
\left(\mean{2B_{i}}\snr{Dw_{i}}^{q} \, dx \right)^{1/q}\le c \left(\mean{4B_{i}}H(x,Dw_{i}) \, dx \right)^{1/p}\;,
\end{flalign}
with $c\equiv c(\data_0, K)$. Back to the estimation of the last term in \rif{62}, Young's inequality gives
\begin{flalign}\label{69}
\mathrm{(I)}&\le ca(x_{i,m})\mean{2B_{i}}(\snr{Dw_{i}}+\snr{Dv_{i}})^{q-1}\snr{Dw_{i}-Dv_{i}} \, dx \nonumber\\
&\le  c\mean{2B_{i}}a(x_{i,m})\snr{Dv_{i}}^{q} \, dx +c\mean{2B_{i}}a(x_{i,m})\snr{Dw_{i}}^{q} \, dx\;,
\end{flalign}
for $c\equiv c(n,\nu,L,p,q)$. We start estimating the last term in \eqref{69} as follows
\begin{eqnarray*}
\mean{2B_{i}}a(x_{i,m})\snr{Dw_{i}}^{q}\, dx &\stackleq{66}  & c\varrho_{i}^{\alpha}\mean{2B_{i}}\snr{Dw_{i}}^{q} \, dx \\
& \stackleq{68} & c\varrho_{i}^{\alpha}\left(\mean{4B_{i}}H(x,Dw_{i}) \, dx\right)^{\frac{q-p}{p}}\mean{4B_{i}}H(x,Dw_{i}) \, dx\\
& \stackrel{\mbox{H\"older}}{\leq} &  c\varrho_{i}^{\alpha}\left(\mean{4B_{i}}[H(x,Dw_{i})]^{1+\sigma} \, dx\right)^{\frac{q-p}{p(1+\sigma)}}\mean{4B_{i}}H(x,Dw_{i}) \, dx \\
&  \le & c\nr{H(\cdot,Dw_i)}_{L^{1+\sigma}(4B_{i})}^{\frac{q-p}{p(1+\sigma)}}\varrho_{i}^{\alpha-\frac{n(q-p)}{p(1+\sigma)}}\mean{4B_{i}}H(x,Dw_{i}) \, dx\\
&  \stackleq{54} & c\nr{H(\cdot,Du)}_{L^{1+\sigma}(4B_{i})}^{\frac{q-p}{p(1+\sigma)}}\varrho_{i}^{\alpha-\frac{n(q-p)}{p(1+\sigma)}}\mean{4B_{i}}H(x,Dw_{i}) \, dx
\\
&  \stackleq{quellaglobale2} & c\varrho_{i}^{\alpha-\frac{n(q-p)}{p(1+\sigma)}}\mean{4B_{i}}H(x,Dw_{i}) \, dx
\\
&  \stackleq{54} & c\varrho_{i}^{\alpha-\frac{n(q-p)}{p(1+\sigma)}}\mean{4B_{i}}H(x,Du) \, dx
\end{eqnarray*}
for $c\equiv c(\data_0, K)$. Letting 
\eqn{kappa1}
$$\kappa_{1}:=\alpha - \frac{n(q-p)}{p(1+\sigma)}> \alpha - n\left(\frac{q}{p}-1\right)\stackrel{\rif{ass4}}{\geq}0$$ and we conclude with
\begin{flalign}\label{70}
\varrho_{i}^{\alpha}\mean{2B_{i}}\snr{Dw_{i}}^{q} \, dx +\mean{2B_{i}}a(x_{i,m})\snr{Dw_{i}}^{q} \, dx \le c\varrho_{i}^{\kappa_{1}}\mean{4B_{i}}H(x,Du) \, dx\;.
\end{flalign}
In order to estimate the remaining term in the last line of \eqref{69}, the one featuring $v_i$, we need to make use of Theorem \ref{T4}. Precisely, we aim to apply it with the choice $\gamma=q/p$ and then combine the outcome with Theorem \ref{T2} for $(2q-p)=q^{2}/p$. We first show that $q^{2}/p$ enters in the range of exponents covered by \eqref{39}. For this, we notice that 
\begin{flalign*}
\frac{n\alpha(\alpha+2n)}{2(n+\alpha)^{2}} < \beta < \alpha \Longrightarrow \frac{q^{2}}{p^{2}}\stackleq{ass4} \left(1+\frac{\alpha}{n}\right)^{2}<\frac{n}{n-2\beta}<\frac{n}{n-2\alpha}\;.
\end{flalign*}
The last quantity in the previous display is meant to be $\infty$ when $\alpha=1$ and $n=2$. Therefore we have that $w_{i}\in W^{1,q^{2}/p}(2B_{i})$ and the reverse inequality holds as a consequence of \rif{quellaglobale2}
\begin{flalign}\label{75}
\left(\mean{2B_{i}}\snr{Dw_{i}}^{q^{2}/p} \, dx\right)^{p/q^{2}}\le c\left(\mean{4B_{i}}H(x,Dw_{i}) \, dx\right)^{1/p}\;,
\end{flalign}
where $c\equiv c(\data, K)$. Estimate \eqref{74} (applied with $\gamma=q/p$ and $a_{0}=a(x_{i,m})$, so that $H_{0}(z)=H(x_{i,m},z)$) reads as
\begin{flalign}\label{76}
\mean{2B_{i}}[H(x_{i,m},Dv_{i})]^{q/p} \, dx \le c\mean{2B_{i}}[H(x_{i,m},Dw_{i})]^{q/p} \, dx\;,
\end{flalign}
for $c\equiv c(n,\nu,L,p,q)$ and with a finite right-hand side. We then argue as follows:
\begin{eqnarray}\label{77}
\mean{2B_{i}}a(x_{i,m})\snr{Dv_{i}}^{q} \, dx & \stackleq{66} &c\varrho_{i}^{\alpha}\mean{2B_{i}}\snr{Dv_{i}}^{q} \, dx \nonumber \\
& \le & c\varrho_{i}^{\alpha}\mean{2B_{i}}[H(x_{i,m},Dv_{i})]^{q/p} \, dx\nonumber \\
&\stackleq{76}  & c\varrho_{i}^{\alpha}\mean{2B_{i}}[H(x_{i,m},Dw_{i})]^{q/p} \, dx\nonumber \\
& \stackleq{66}& c\varrho_{i}^{\alpha}\mean{2B_{i}}\snr{Dw_{i}}^{q} \, dx +c\varrho_{i}^{\alpha(1+q/p)}\mean{2B_{i}}\snr{Dw_{i}}^{q^{2}/p} \, dx\nonumber\\
&\stackleq{70} & c\varrho_{i}^{\kappa_{1}}\mean{4B_{i}}H(x,Dw_{i}) \, dx +c\varrho_{i}^{\alpha(1+q/p)}\mean{2B_{i}}\snr{Dw_{i}}^{q^{2}/p} \, dx\;,
\end{eqnarray}
with $c\equiv c(\data_0,K)$. Notice that, upon defining 
\eqn{kappa2}
$$
\kappa_{2}:=\alpha\left(1+\frac{q}{p}\right)-\frac{n}{1+\sigma}\left[\left(\frac{q}{p}\right)^{2}-1\right]>
\left(1+\frac{q}{p}\right)\left[\alpha - n\left(\frac{q}{p}-1\right)\right]\stackrel{\rif{ass4}}{\geq} 0\;,
$$
we have
\begin{eqnarray}\label{78}
\varrho_{i}^{\alpha(1+q/p)}\mean{2B_{i}}\snr{Dw_{i}}^{q^{2}/p} \, dx & \stackleq{75} &c\varrho_{i}^{\alpha(1+q/p)}\left(\mean{4B_{i}}H(x,Dw_{i}) \, dx\right)^{q^{2}/p^{2}-1}\mean{4B_{i}}H(x,Dw_{i}) \, dx\nonumber \\
& \stackrel{\mbox{H\"older}}{\leq} &c\varrho_{i}^{\alpha(1+q/p)}\left(\mean{4B_{i}}[H(x,Dw_{i})]^{1+\sigma} \, dx\right)^{\frac{q^{2}/p^{2}-1}{1+\sigma}}\mean{4B_{i}}H(x,Dw_{i}) \, dx\nonumber \\
& \leq &c\varrho_{i}^{\alpha(1+q/p)-n\frac{q^{2}/p^{2}-1}{1+\sigma}}\nr{H(\cdot,Dw_i)}_{L^{1+\sigma}(4B_{i})}^{q^{2}/p^{2}-1}\mean{4B_{i}}H(x,Dw_{i}) \, dx\nonumber \\
&\stackleq{54} &  c\varrho_{i}^{\kappa_{2}}\nr{H(\cdot,Du)}_{L^{1+\sigma}(4B_{i})}^{q^{2}/p^{2}-1}\mean{4B_{i}}H(x,Dw_{i}) \, dx\nonumber \\ & \le & c\varrho_{i}^{\kappa_{2}}\mean{4B_{i}}H(x,Dw_{i}) \, dx \nonumber
\\ & \stackleq{54} & c\varrho_{i}^{\kappa_{2}}\mean{4B_{i}}H(x,Du) \, dx\;,
\end{eqnarray}
where $c\equiv c(\data_0,K)$. 
Collecting estimates \eqref{77}, \eqref{78} yields
$$
\mean{2B_{i}}a(x_{i,m})\snr{Dv_{i}}^{q} \, dx  \leq  c\varrho_{i}^{\kappa_{2}}\mean{4B_{i}}H(x,Du) \, dx\;.
$$
Using this last estimate and \rif{70} in \rif{69}, we again obtain
\begin{flalign*}
\mathrm{(I)}\le c\varrho_{i}^{\kappa_{1}}\mean{4B_{i}}H(x,Du) \, dx\;\;,
\end{flalign*}
with $c\equiv c(\data_0, K)$. We have also used that $\kappa_1 < \kappa_2$ (compare \rif{kappa1} and \rif{kappa2}) and that $\varrho_i \leq 1$. Finally, this last estimate and \eqref{62} lead to
\begin{flalign}\label{79}
\mean{2B_{i}}\snr{V_{p}(Dv_{i})-V_{p}(Dw_{i})}^{2}+a(x_{i,m})\snr{V_{q}(Dv_{i})-V_{q}(Dw_{i})}^{2} \, dx \le c_{*}\varrho_{i}^{\kappa_{1}}\mean{4B_{i}}H(x,Du) \, dx\;,
\end{flalign}
for $c_{*}\equiv c_{*}(\data_0, K)$. We also observe that, by first using \eqref{70} and then $\eqref{54}$ we obtain the following analog of \eqref{81pre}:
\begin{flalign}\label{81}
\mean{2B_{i}}H(x_{i,m},Dw_{i}) \, dx \le c\mean{4B_{i}}H(x,Du) \, dx\;,
\end{flalign}
where this time it is $c\equiv c(\data_0, K)$.

\emph{Step 7: Matching the two phases and comparison estimates}. Summarizing the content of \eqref{65} and \eqref{79}, we have in both cases \eqref{pqp} and \eqref{pp} that the following inequality holds:
\begin{flalign*}
\mean{2B_{i}}&\snr{V_{p}(Dw_{i})-V_{p}(Dv_{i})}^{2}+a(x)\snr{V_{q}(Dw_{i})-V_{q}(Dv_{i})}^{2} \, dx \\
& \le\mean{2B_{i}}\snr{V_{p}(Dw_{i})-V_{p}(Dv_{i})}^{2}+a(x_{i,m})\snr{V_{q}(Dw_{i})-V_{q}(Dv_{i})}^{2} \, dx\le \left(\frac{\bar{c}}{K}+c_{*}\varrho_{i}^{\kappa_1}\right)\mean{4B_{i}}H(x,Du) \, dx\;\;,
\end{flalign*}
with $\bar{c}\equiv \bar{c}(n,\nu,L,p,q)$ and $c_{*}\equiv c_{*}(\data_0, K)$, where $K\geq 4$ is still to be chosen and $\kappa_1$ has been defined in \rif{kappa1}. This last estimate and \eqref{56} (recall that the $\varepsilon$ in \eqref{56} is arbitrary) then give
\begin{flalign}\label{82}
\mean{2B_{i}}&\snr{V_{p}(Dv_{i})-V_{p}(Du)}^{2}+a(x)\snr{V_{q}(Dv_{i})-V_{q}(Du)}^{2} \, dx \nonumber \\
& \le \left(2\varepsilon +2c_{*}r^{\kappa_{1}}+\frac{2\bar{c}}{K}\right)\mean{4B_{i}}H(x,Du) \, dx+2c_{\varepsilon}\mean{4B_{i}}H(x,F) \, dx\;,
\end{flalign}
for all $\varepsilon \in (0,1)$, where again it is $\bar{c}\equiv \bar{c}(n,\nu,L,p,q)$ and $c_{*}\equiv c_{*}(\data_0, K)$, where $K\geq 4$, while $c_{\varepsilon}\equiv c_{\varepsilon}(n,p,q,\nu,L,\varepsilon)$ and $\kappa_1>0$ has been defined in \rif{kappa1}. Here we also used the fact that $\varrho_{i}\le r$ (see Step 1). We now adopt the notation
$$S(\varepsilon,r,K,M):=2\varepsilon+2c_{*}r^{\kappa_1}+\frac{2\bar{c}}{K}+\frac{2c_{\varepsilon}}{M}$$
and, using the information contained in \cite[$(4.14)_{2}$]{CM3} in \eqref{82}, we establish that for every $K\ge 4$ the estimate \begin{flalign}\label{83}
\mean{2B_{i}}\snr{V_{p}(Dv_{i})-V_{p}(Du)}^{2}+a(x)\snr{V_{q}(Dv_{i})-V_{q}(Du)}^{2} \, dx\le S(\varepsilon,r,K,M)\lambda
\end{flalign}
holds for all the balls $B_{i}$ from the covering displayed in \cite[(4.11)]{CM3}. Notice that at this stage this in an estimate which is phase-independent: no matter of which among \eqref{pqp} and \eqref{pp}, we have that occurs \rif{83} holds in any case. Moreover, we are still free to choose $K\ge 4$. 

\emph{Step 8: The two phases at a different threshold}. The goal here is to show that the estimate
\begin{flalign}\label{84}
\mean{2B_{i}}H(x_{i,m},Dv_{i}) \, dx =\mean{2B_{i}}\snr{Dv_{i}}^{p}+a(x_{i,m})\snr{Dv_{i}}^{q} \, dx \le c\lambda,
\end{flalign}
holds for $c\equiv c(\data_0)$. For this we simply consider the two alternatives
\begin{flalign}\label{85}
\inf_{x \in 2B_{i}}a(x)>10[a]_{0,\alpha}\varrho_{i}^{\alpha}\quad \mathrm{and} \quad \inf_{x\in 2B_{i}}a(x)\le 10[a]_{0,\alpha}\varrho_{i}^{\alpha}\;,
\end{flalign}
which are nothing but \eqref{pqp} and \eqref{pp} with $K=10$ respectively. We start considering the case in which the first inequality in \eqref{85} holds; by using \eqref{59} first and then \eqref{81pre} we get that
\begin{flalign}\label{86}
\mean{2B_{i}}H(x_{i,m},Dv_{i}) \, dx \le c\mean{4B_{i}}H(x,Du) \, dx\;,
\end{flalign}
for $c\equiv c(n,\nu,L,p,q)$. At this stage \eqref{84} follows recalling the setting of \cite[$(4.14)_{2}$]{CM3}. In case the second inequality in $\eqref{85}$ holds, we similarly use \eqref{59} again and \eqref{81} with $K=10$, that renders \eqref{86} with a constant $c\equiv c(\data_0)$. Therefore, \eqref{84} follows in every case.

\emph{Step 9: A priori estimates for $Dv_{i}$}. The same as in the proof of \cite[Theorem 1.1]{CM3}.

\emph{Step 10: Estimates involving level sets}. The same as in the proof of \cite[Theorem 1.1]{CM3}.

\emph{Step 11: Integration and conclusion}. The same as in the proof of \cite[Theorem 1.1]{CM3} and this concludes the proof of Theorem \ref{cz}. We only remark that the peculiar dependence on the constants on $\data_0$ comes from the estimates in Steps 2-8, and finally reflects in \rif{final-estimate}.\\\\
\textbf{Acknowledgments.} This work was supported by the \emph{Engineering and Physical Sciences Research Council} $[EP/L015811/1]$. We thank the referee for a careful reading of the first version of the manuscript.

\end{document}